\documentclass[a4paper,11pt]{article}
\usepackage[utf8]{inputenc}
\usepackage[T1]{fontenc}

\usepackage{indentfirst}
\usepackage{enumitem}
\usepackage{amsmath}
\usepackage{amssymb,amsmath,amsthm}
\usepackage{ebproof}
\usepackage{aliascnt}
\usepackage{comment}
\usepackage{hyperref}
\usepackage[capitalise]{cleveref}
\usepackage{ragged2e}
\usepackage{adjustbox}
\usepackage [autostyle, english = american]{csquotes}
\MakeOuterQuote{"}

\usepackage[style=alphabetic,maxalphanames=4,minalphanames=4]{biblatex}
\usepackage[affil-it]{authblk}

\usepackage{tikz-cd}
\usetikzlibrary{decorations.markings,intersections}
\tikzset{shorten <>/.style={shorten >=#1,shorten <=#1}}
\tikzset{every picture/.prefix code=\DisableQuotes}

\tikzset{%
scalearrow/.style n args={3}{
  decoration={
    markings,
    mark=at position (1-#1)/2*\pgfdecoratedpathlength
      with {\coordinate (#2);},
    mark=at position (1+#1)/2*\pgfdecoratedpathlength
      with {\coordinate (#3);},
    },
  postaction=decorate,
  }
}

\newcommand{\catname}[1]{\mathbf{#1}}

\usepackage{quiver}
\usepackage[mathcal]{euscript}
\usepackage{ stmaryrd }
\usepackage{ dsfont }
\usepackage{upgreek}
\usepackage{ mathrsfs }

\renewenvironment{abstract}
 {\small
  \begin{center}
  \bfseries \abstractname\vspace{-.5em}\vspace{0pt}
  \end{center}
  \list{}{%
    \setlength{\leftmargin}{2mm}
    \setlength{\rightmargin}{\leftmargin}%
  }%
  \item\relax}
 {\endlist}
 
 \let\oldtheorem\newtheorem
\RenewDocumentCommand{\newtheorem}{s m o m O{}}{%
\IfBooleanTF{#1}%
{\oldtheorem{#2}{#4}}%
{\IfNoValueTF{#3}{\oldtheorem{#2}{#4}[#5]}%
{\newaliascnt{#2}{#3}%
\oldtheorem{#2}[#2]{#4}%
\aliascntresetthe{#2}}}}

\newtheorem{theorem}{Theorem}[section]
\newtheorem{proposition}[theorem]{Proposition}
\newtheorem{lemma}[theorem]{Lemma}
\newtheorem{corollary}[theorem]{Corollary}
\newtheorem{conjecture}[theorem]{Conjecture}

\theoremstyle{definition}
\newtheorem{definition}{Definition}[section]

\theoremstyle{remark}
\newtheorem{remark}{Remark}[section]
\newtheorem{example}[remark]{Example}

\title{Internal languages of locally cartesian closed $(\infty,1)$-categories}
\author{El Mehdi Cherradi}
\affil{IRIF - CNRS - Universit\'e Paris Cit\'e \\MINES ParisTech - Universit\'e PSL}
\date{}

\begin{document}

\maketitle

\begin{abstract}
 We establish a DK-equivalence between the relative category of $\pi$-tribes and the relative category of locally cartesian closed quasicategories. From this follows one of the internal languages conjectures: Martin-Löf type theory with dependent sums, intensional identity types, and dependent products satisfying functional extensionality is the internal language of locally cartesian closed $(\infty,1)$-categories.
\end{abstract}

\tableofcontents

\newpage 

\section*{Introduction}
\addcontentsline{toc}{section}{Introduction}

The connection between type theory and category theory has been fruitfully studied at various levels. Indeed, on the one hand, it is possible to rely on either of these theories as a foundational framework for logic and mathematics where, in particular, the other theory can be developed. On the other hand, and in relation with the previous observation, categorical logic identifies the elementary constructions in category theory that correspond to logical constructions such as the usual type constructors of Martin-Löf type theory.

A well-known result that subsumes an important fragment of logic is the equivalence between models of MLTT with dependent sums and products and locally cartesian closed categories, as established by Hofmann in \cite{hofmann1994}. For this reason, one could argue the link between (extensional) type theory and ($1$-)category theory is fairly well-understood.

However, in recent years, there has been an important development of homotopy theory and homotopy-enabled versions of MLTT and category theory, referred to as homotopy type theory (HoTT) and $(\infty,1)$-category theory, respectively. On the logical side, this offers foundational frameworks much betterF adapted to carrying out arguments where notions of sameness other than equality are considered. In particular, synthetic mathematics dealing with homotopy theory is easier to express and reason with in such frameworks, as they provide a language that abstracts away the technical aspects of homotopy theory and ensures inherent compatibility of the logical constructions with homotopy. This is also analogous to how usual type theory facilitates mathematical reasoning compared to arguments carried out in a purely set-theoretic language.

Being a younger research area, the connection between homotopy type theory and $(\infty,1)$-category theory is still, in many respects, conjectural, although some results are well expected based on their extensional/$1$-categorical counterpart, and important progress has been made recently.

This document builds up on the important work of Kapulkin and Szumiło, whose paper \cite{ks2019internal} has paved the way towards an analogue of Hofmann's result within the realm of homotopy, by establishing an equivalence between models of intensional MLTT with dependent sums and $(\infty,1)$-categories with finite limits.

\subsection*{Statement of the conjecture}

This paper's starting point is the collection of results obtained in \cite{ks2019internal} (more precisely, Theorem 9.10) and \cite{kapulkin2015locally} (Theorem 5.3). The $\infty$-categorical localization functor 
\begin{equation}
\label{hoi}
\mathbf{Ho}_\infty : \catname{weCat} \to \catname{QCat}
\end{equation}
which maps every relative category to its underlying quasicategory, can be implemented in several ways, for instance, by applying the simplicial nerve to a fibrant replacement of the hammock localization. 

The internal language conjecture for locally cartesian closed $(\infty,1)$-categories, formulated in \cite{kl2018homotopy}, can be stated as follows:

\begin{conjecture}
\label{ilcj}
 The functor $$\mathbf{Ho}_\infty: \catname{CompCat}_{\Sigma, \Pi_\text{ext}, \text{Id}} \to \catname{QCat}_{lcc}$$
 
 induces a DK-equivalence. 
\end{conjecture}

Our goal, in this document, is to establish the conjecture, which can be rephrased by saying that the internal language of locally cartesian closed quasicategories is a dependent type theory with dependent sums, dependent products that are extensional and (intensional) identity types.
The functor can be written as the composite
 $$\mathbf{Ho}_\infty: \catname{CompCat}_{\Sigma, \Pi_\text{ext}, \text{Id}} \to \mathbf{Trb}_\pi \to \catname{QCat}_{lcc}$$
 where the first component being a DK-equivalence is easier to establish. Therefore, we will mostly be concerned with the comparison between $\pi$-tribes (in the sense of Joyal's notes \cite{joyal2017notes}) and locally cartesian closed quasicategories.

\subsection*{Outline}

Overall, the strategy of the proof is as follows:
\begin{itemize}
 \item Decompose the problem into several steps, introducing an intermediate category between the category of $\pi$-tribes and the category of lcc quasicategories: a category of tribes equivalent to $\pi$-tribes, where the morphisms preserve the dependent product "loosely", that is up to equivalence.
 \item Reduce the problem to that of establishing DK-equivalences between various relative categories that admit a fibration category structure. This is useful to make use of Cisinski's characterization of DK-equivalence between fibration categories as exact functors that induce an equivalence of categories between the underlying homotopy categories (see \cref{cisinski_dke}). This can be thought, in a sense, as a fibrant replacement of these relative categories. We will get a diagram as below,

\[\begin{tikzcd}[ampersand replacement=\&]
	\textcolor{rgb,255:red,204;green,51;blue,51}{\catname{scTrb_\pi^p}} \&\& {\catname{scTrb_\pi}} \&\& {\mathbf{Trb_\pi}} \\
	\\
	\&\& \textcolor{rgb,255:red,204;green,51;blue,51}{\catname{scTrb_{\pi,\sim}}} \&\& {\mathbf{Trb_{\pi,\sim}}} \\
	\\
	\&\& \textcolor{rgb,255:red,204;green,51;blue,51}{\mathbf{scTrb}} \&\& {\mathbf{Trb}}
	\arrow["\sim"{description}, from=1-1, to=1-3]
	\arrow["\sim"{description}, from=1-3, to=1-5]
	\arrow["\sim"{description}, from=1-3, to=3-3]
	\arrow["\sim"{description}, from=1-5, to=3-5]
	\arrow["\sim"{description}, from=3-3, to=3-5]
	\arrow[from=3-3, to=5-3]
	\arrow[from=3-5, to=5-5]
	\arrow["\sim"{description}, from=5-3, to=5-5]
\end{tikzcd}\]
where the categories in red are replacements for the "naturally occurring" categories on the right.
 
 \item Show that any functor between $\pi$-tribes that preserves the dependent product up to equivalence can be "factored" as a span of functors that preserve it up to isomorphism. This is the key idea to establish a DK-equivalence between $\pi$-tribes and "loose" $\pi$-tribes (in the sense above). 
\end{itemize}

The first section studies the tribe $P\mathcal{T}$, for $\mathcal{T}$ a ($\pi$-)tribe, which is a the natural candidate for defining a path object in various categories of tribes. The second section defines several categories of tribes together with appropriate "replacements" in the form of DK-equivalent categories of tribes enjoying the structure of fibration categories. In the third section, we discuss the simple but key idea that enables a rigidification process connecting functors that preserve loosely (i.e, up to weak equivalence) exponentials with functors that preserve it up to isomorphism. Finally, the fourth section uses the previous tools to wrap up and establish the conjecture (\cref{ilcj}). 

\subsection*{Acknowledgement}

The author would like to thank Eric Finster, Jonas Frey, Paul-André Melliès and Carlos Simpson for helpful discussions and feedback regarding this material. Special thanks to Michael Shulman for giving insightful comments, pointing out mistakes and making suggestions on drafts of this document.

In the first version of this document, a full proof of the internal language conjecture (\cref{ilcj}) for locally cartesian closed $(\infty,1)$-categories had already been claimed. This turned out to be incorrect, as a subtle problem related to the definition of the considered category of semi-cubes had been overlooked; this is briefly explained at the beginning of Section 2. This problem was noticed by Evan Cavallo, Christian Sattler and Chaitanya Leena Subramaniam, whom the author would like to thank for sharing their observation.

\section{The canonical path tribe}

The more specific notion of $\pi$-tribe (Definition 3.8.1 in \cite{joyal2017notes}) is central here:

\begin{definition}
 A tribe $\mathcal{T}$ is a $\pi$-tribe if every fibration $p : E \to A$ admits an internal product (or dependent product) $\Pi_f p$ along every fibration $f~:~A \to B$, such that the structure map $\Pi_f p$ (with codomain $B$) is a fibration, and if the induced functor between fibrant slices $\Pi_f : \mathcal{T}(A) \to \mathcal{T}(B)$ preserves anodyne maps.
\end{definition}

Fibration categories provide a reasonable setting for investigating possible variations $H : \mathcal{F}_0 \to \mathcal{F}_1$ of the $\mathbf{Ho}_\infty$ functor defined in \eqref{hoi}. Indeed, Cisinski established in \cite{cisinski2010invariance} the following key result, around which the strategy of our proof revolves: 

\begin{theorem}[Cisinski]
\label{cisinski_dke}
 Given fibration categories $\mathcal{F}_0$ and $\mathcal{F}_1$, as well as an exact functor $H : \mathcal{F}_0 \to \mathcal{F}_1$, the following are equivalent:
 \begin{itemize}
  \item $H$ is a DK-equivalence.
  \item $\mathbf{Ho}(H) : \mathbf{Ho}(\mathcal{F}_0) \to \mathbf{Ho}(\mathcal{F}_1)$ is an equivalence of categories.
  \item $H$ satisfies the following two \textit{approximations properties}:
  \begin{enumerate}
   \item[(AP1)] $H$ reflects weak equivalences.
   \item[(AP2)] For every objects $x_0 \in \mathcal{F}_0$ and $y_1 \in \mathcal{F}_1$, and every morphisms $y_1 \to H(x_0)$ in $\mathcal{F}_1$, there exists a commutative square in $\mathcal{F}_1$,

\[\begin{tikzcd}[ampersand replacement=\&]
	{y_1} \&\& {H(x_0)} \\
	\\
	{y'_1} \&\& {H(y_0)}
	\arrow[from=1-1, to=1-3]
	\arrow["\sim"{description}, from=3-1, to=1-1]
	\arrow["\sim"{description}, from=3-1, to=3-3]
	\arrow["{H(f)}"', from=3-3, to=1-3]
\end{tikzcd}\]
   with $f : y_0 \to x_0$ an arrow in $\mathcal{F}_0$, and where the indicated arrows are weak equivalences.
  \end{enumerate}
 \end{itemize}
\end{theorem}

This direction is supported by Szumilo's key observation in \cite{szumilo2016homotopy} that the category $\mathbf{FibCat}$ of fibration categories and exact functors between them is itself a fibration category, with $\mathbf{W}$ the class of exact functors which are DK-equivalences.
The fact that $\mathbf{FibCat}$ can be endowed with a fibration category structure relies on an appropriate definition of the class $\mathbf{F}$ of fibrations and a clever construction of the fibration category $P\mathcal{F}$ defining the path object associated with a general fibration category $\mathcal{F}$.
Given two objects $x$ and $y$ in $\mathcal{F}$, $x$ and $y$ are equivalent, that is connected by a zig-zag of weak equivalences, if and only if they are connected by a zig-zag of length two $\bullet \leftarrow \bullet \rightarrow \bullet$. Hence, a natural intuition is that the subcategory $Q\mathcal{F}$ of  $\mathcal{F}^{\bullet \leftarrow \bullet \rightarrow \bullet}$ formed by those spans where both legs are weak equivalences captures the notion of homotopy inside the fibration category $\mathcal{F}$, and is thus not far from defining a path object for $\mathcal{F}$.

The construction of the path tribe $P\mathcal{F}$ is a key component of the proof, in \cite{szumilo2016homotopy}, that $\mathbf{FibCat}$ enjoys a fibration category structure.
The category $P\mathcal{F}$ of Reedy fibrant diagrams in $Q\mathcal{F}$ with the pointwise notion of weak equivalences has a fibration category structure, which does not provide a path object for $\mathcal{F}$ in general due to the lack of a "reflexivity" exact functor $\mathcal{F} \to P\mathcal{F}$.
In \cite{szumilo2016homotopy}, this hindrance is overcome by introducing a slight variation  on the Reedy structure $P\mathcal{F}$ such that there exists a reflexivity functor $\mathcal{F} \to P\mathcal{F}$ that, together with the two projections $\pi_0, \pi_1 : P\mathcal{F} \to \mathcal{F}$, provides a factorization of the diagonal $\Delta_f : \mathcal{F} \to \mathcal{F} \times \mathcal{F}$.
The reflexivity functor $\mathcal{F} \to P\mathcal{F}$ is provided, in fact, by the diagonal mapping $x \mapsto x \leftarrow x \rightarrow x$ taking $x$ to the constant span.

There is a similar problem arising when considering the category $\mathbf{Trb}$ of tribes and morphisms of tribes (exact functors that preserve anodyne maps). Another way to work around this is to restrict $\mathbf{Trb}$ to those tribes that admit a functorial path object construction $x \mapsto \overline{x}$. The functorial path object, with its two projections $p_i : \overline{x} \to x$, is then used to define a functor $\mathcal{T} \to P\mathcal{T}$ between tribes, where $P\mathcal{T}$ is defined as the category of Reedy fibrant spans whose legs are weak equivalences. Note that $P\mathcal{T}$ is equipped with the more standard notion of Reedy fibration (rather than Szumilo's variation) thanks to the fact that we consider the mapping $x \mapsto \overline{x}$. Every (semi-)simplicial tribe comes equipped with such a functorial construction, given by cotensoring $x$ with $\Delta^1$. Supplying a construction to replace a tribe $\mathcal{T}$ by a semi-simplicial DK-equivalence tribe $\mathbf{\mathcal{T}}$ enables one to switch from the category $\mathbf{Trb}$ to the category $\mathbf{scTrb}$ of semi-cubical tribes and semi-cubical exact functors, and its "canonical" fibration category structure. This is the approach followed in \cite{ks2019internal}.

We seek to use a similar approach, where the construction $\mathcal{T} \mapsto P\mathcal{T}$ plays a central role. We begin by recalling the definition from Section 4 of \cite{ks2019internal}.

A homotopical category is a category $C$ together with a class of maps $W$, called the weak equivalences. A homotopical functor between homotopical categories is a functor that sends weak equivalences to weak equivalences.

\begin{definition}
\label{kscpo}
 We define $\mathbf{Sp}_w$ to be the "homotopical span" category, that is the following category
 $$\bullet \stackrel{\sim}{\leftarrow} \bullet \stackrel{\sim}{\rightarrow} \bullet$$
 where both maps are weak equivalences. $\mathbf{Sp}_w$ admits a Reedy category structure (which is an inverse one): the apex has degree $1$, and the two other objects have degree $0$.
 
 For $\mathcal{T}$ a tribe, we write $P\mathcal{T}$ for the category of Reedy fibrant diagrams from $\mathbf{Sp}_w$ to $\mathcal{T}$.
\end{definition}
This category of diagrams also inherits a tribe structure. An elementary way to prove this is to rely on an alternative definition of $P\mathcal{T}$ based on the construction $\mathcal{T}^{(1)}$ and  $\mathcal{T}^{(\wedge)}$ from Section 1.8 of \cite{joyal2017notes}.

\begin{definition}
\label{cpo0}
 Let $\mathcal{T}$ be a tribe. Define $\mathcal{T}^{(1)}$ to be the full subcategory of the arrow category $\mathcal{T}^{\rightarrow}$ whose objects are the fibrations. Equipped with the class of Reedy fibrant squares as its notion of fibrations, $\mathcal{T}^{(1)}$ enjoys a tribe structure.
 
 The category $\mathcal{T}^{(\wedge)}$ is defined by the following pullback square in $\catname{Cat}$:

\[\begin{tikzcd}
	{\mathcal{T}^{(\wedge)}} && {\mathcal{T}^{(1)}} \\
	\\
	{\mathcal{T} \times \mathcal{T}} && {\mathcal{T}}
	\arrow["C"', from=3-1, to=3-3]
	\arrow[two heads, from=1-1, to=3-1]
	\arrow["{\partial_1}", two heads, from=1-3, to=3-3]
	\arrow[from=1-1, to=1-3]
	\arrow["\lrcorner"{anchor=center, pos=0.125}, draw=none, from=1-1, to=3-3]
\end{tikzcd}\]
where $C : \mathcal{T} \times \mathcal{T} \to \mathcal{T}$ is the cartesian product functor and $\mathcal{T}^{(1)}$ is the tribe of fibrations as introduced in Section 1.7 of \cite{joyal2017notes}.
\end{definition}

Note that this pullback in $\mathbf{Cat}$ inherits a canonical clan structure, by Lemma 1.4.8 in \cite{joyal2017notes}, which can be promoted to a tribe structure. Moreover, it also yields a pullback square in $\catname{Trb}$. 

\begin{remark}
 The anodyne maps in $\mathcal{T}^{(1)}$ are the pointwise anodyne morphisms, and the anodyne maps in $\mathcal{T}^{(\wedge)}$ are the component-wise anodyne morphisms (with respect to the defining pullback in the previous definition). 
\end{remark}

\begin{lemma}
\label{cpo}
 The category $P\mathcal{T}$ coincides with the sub-tribe of $\mathcal{T}^{(\wedge)}$ whose objects consist of spans $x \to y \times z$ such that $x \to y$ and $x \to z$ are trivial fibrations. 
\end{lemma}
\begin{proof}
 The Reedy fibrancy condition on the diagrams corresponds to the definition of $\mathcal{T}^{(\wedge)}$ (i.e., diagrams $y \leftarrow x \rightarrow z$ such that $x \to y \times z$ is a fibration). The subcategory considered is the one corresponding to those diagrams that are homotopical (i.e both $x \to y$ and $x \to z$ are weak equivalences).
\end{proof}

\begin{proposition}
\label{fonctp}
 The construction $\mathcal{T} \mapsto P\mathcal{T}$ defines a limit preserving endofunctor of $\mathbf{Trb}$.
 
\end{proposition}

\begin{proof}
 First note that $\mathcal{T} \mapsto \mathcal{T}^{(1)}$ defines a limit preserving functor. Indeed the arrow-category mapping $\mathcal{T} \mapsto \mathcal{T} ^{\rightarrow}$ is such a functor. Moreover, taking the fibrations in a pullback square along an isofibration (in $\mathbf{Cat}$) to be component-wise

\[\begin{tikzcd}
	{\mathcal{T}} && {\mathcal{T}_1} \\
	\\
	{\mathcal{T}_2} && {\mathcal{T}_0}
	\arrow[from=3-1, to=3-3]
	\arrow[from=1-3, to=3-3]
	\arrow[from=1-1, to=3-1]
	\arrow[from=1-1, to=1-3]
	\arrow["\lrcorner"{anchor=center, pos=0.125}, draw=none, from=1-1, to=3-3]
\end{tikzcd}\]
 turns $\mathcal{T}$ into a tribe, and makes the defining pullback a pullback in $\mathbf{Trb}$.

 By definition of $\mathcal{T}^{(\wedge)}$ as a pullback, it follows directly that $\mathcal{T} \mapsto \mathcal{T}^{(\wedge)}$ defines a limit preserving functor. Since, in the previous pullback square, a fibration $x \to y \times z$ in $\mathcal{T}$ has components $x \to y$ and $x \to z$ trivial fibrations if and only if the projected spans $x_1 \to y_1 \times z_1$ and $x_2 \to y_2 \times z_2$ have this property (i.e, the inherited notion of weak equivalences in $\mathcal{T}$ being component-wise), we indeed have a limit preserving functor $\mathcal{T} \mapsto P\mathcal{T}$.
 \end{proof}

 We would like to establish the following key proposition.

\begin{proposition}
\label{piP}
 If $\mathcal{T}$ is a $\pi$-tribe, then so is $P\mathcal{T}$, and the projections $p_0, p_1 : P\mathcal{T} \to \mathcal{T}$ are $\pi$-closed. 
\end{proposition}

For this purpose, we use \cref{piR}, which is a reformulation of a more general result in \cite{kapulkin2021homotopical}. Recall that a homotopical inverse category is an inverse category with a class of weak equivalences containing all identity morphisms and enjoying the $2$-out-of-$6$ property. A diagram $I \to \mathcal{T}$, whose shape is a homotopical inverse category $I$, is a functor $D : I \to \mathcal{T}$ mapping weak equivalences to weak equivalences. In Proposition 5.13 of \cite{kapulkin2021homotopical}, Kapulkin and Lumsdaine prove that for a category with attributes $\mathbf{C}$ admitting $Id$-type and extensional $\Pi$-type (satisfying moreover the $\eta$-rule), the category $\mathbf{C}^I$ of homotopical diagrams, with the notion of Reedy types they define, is a category with attributes admitting the same logical structure. We would like to make use of this result in the context of tribes. To do so, we can either rephrase the given proof in the language of tribes or use the connection between tribes and categories with attributes to transfer their result. We choose to use this second method. Hence, we need \cref{compcat_pitribe} below.

 First, let $T : \mathbf{CompCat}_{\Sigma, id} \to \mathbf{Trb}$ and $C : \mathbf{Trb} \to \mathbf{CompCat}_{\Sigma, id}$ be the functors defined in Section 9 of \cite{ks2019internal}. Explicitly, recall that for $\mathbf{C}$ an object in $\mathbf{CompCat}_{\Sigma, id}$, the tribe $T \mathbf{C}$ has the same underlying category and the notion of fibration given as (finite) composites of context projections. For $\mathcal{T}$ a tribe, $C \mathcal{T}$ is a (full) comprehension category constructed from the codomain projection $\mathcal{T}^{\rightarrow_\mathbf{fib}} \to \mathcal{T}$, where $\mathcal{T}^{\rightarrow_\mathbf{fib}}$ is the full subcategory of $\mathcal{T}^{\rightarrow}$ spanned by the fibrations.

\begin{lemma}
\label{compcat_pitribe}
With the notations above, if $\mathcal{T}$ is a $\pi$-tribe, then $C \mathcal{T}$ admits $\Pi$-types satisfying the $\eta$-rule and with function extensionality $\mathbf{ext} : \Pi_f P X \to P\Pi_f X$ mapping the reflexive homotopy to the reflexive path (i.e. in the image of $\iota : \Pi_f X \to P\Pi_f X$). Conversely, if $\mathbf{C}$ admits such $\Pi$-types, then $T \mathbf{C}$ is a $\pi$-tribe.
\end{lemma}

\begin{proof}
 For the first part, the definition of an internal product in the tribe $\mathcal{T}$ is precisely designed so that its universal property can be used to construct $\Pi$-types, which moreover satisfy the $\eta$-rule. Function extensionality follows from preservation of anodyne maps, so that we have a lift in the diagram below, satisfying the extra assumption on reflexiveness, which is a weak equivalence by $2$-out-of-$3$.

\[\begin{tikzcd}[ampersand replacement=\&]
	{\Pi_f B} \&\& {P\Pi_ fB} \\
	\\
	{\Pi_f PB} \&\& {\Pi_f (B\times_A B)\simeq \Pi_f B \times_A \Pi_f B}
	\arrow["{\iota_{\Pi_fB}}", from=1-1, to=1-3]
	\arrow["{\Pi_f \iota_B}"', from=1-1, to=3-1]
	\arrow[from=1-3, to=3-3]
	\arrow["{\mathbf{ext}}", dashed, from=3-1, to=1-3]
	\arrow["{\Pi_f(<p_0,p_1>)}"', from=3-1, to=3-3]
\end{tikzcd}\]
Here, $f : A \to A'$ is a fibration, $PX$ denotes a path object for $X$, which comes with the reflexivity map $\iota_X : X \to PX$.

For the converse, we can argue as in Lemma 5.5 of \cite{kapulkin2015locally}. Namely, the evaluation map $\epsilon : \Gamma, A, \Pi_f B (\simeq B \times_A \Pi_f B) \to \Gamma,A,B$ comes from the morphism $\mathbf{app}_{A,B}$ supplied by the $\Pi$-type structure. The $\eta$-rule implies the universal property expected from the internal product, i.e. the evaluation being cofree with respect to the functor $f^* : \mathcal{T}_{/A'} \to \mathcal{T}_{/A}$ (and not just the functor between fibrant slices   $f^* : \mathcal{T}(A') \to \mathcal{T}(A)$). Indeed, given a map $u : C \times_\Gamma A \to B$ for some type $\Gamma, C$, we have a section $\lambda : \Gamma, C \to \Gamma, \Pi_f B \times_\Gamma C$ of the projection $\Gamma, \Pi_f B \times_\Gamma C \to  \Gamma, C$, satisfying $\epsilon \circ (\lambda \times_\Gamma A) = u$ by definition of the $\eta$-rule.

We still need to check that the internal product functor $\Pi_f : \mathcal{T}(A) \to \mathcal{T}(A')$ preserves anodyne maps. To see this, we will argue as in Lemmas 4.3.4 and 4.3.5 in \cite{joyal2017notes}, using the characterization of the anodyne maps as the strong deformation retracts. Explicitly, if $u: X \to Y$ is a map in $\mathcal{T}(A)$ which is anodyne, it is a strong deformation retract, so there exists a map $r: Y \to X$ such that $r \circ u = id_X$ and $u \circ r$ is homotopic to $id_Y$. We can take a lift $Pu : PX \to PY$, where $PX$ and $PY$ are the path objects provided by the identity types of the category with attributes $\mathbf{C}$, as in the diagram below on the left, and take a homotopy $h : Y \to PY$ as on the right.

\[\begin{tikzcd}[ampersand replacement=\&]
	X \&\& PX \&\& X \& PX \& PY \\
	Y \\
	PY \&\& {*} \&\& Y \&\& {Y \times_A Y}
	\arrow["{\iota_X}", from=1-1, to=1-3]
	\arrow["u"', from=1-1, to=2-1]
	\arrow[from=1-3, to=3-3]
	\arrow["{\iota_X}", from=1-5, to=1-6]
	\arrow["u"', from=1-5, to=3-5]
	\arrow["Pu", from=1-6, to=1-7]
	\arrow[from=1-7, to=3-7]
	\arrow["{\iota_Y}"', from=2-1, to=3-1]
	\arrow["Pu"', dashed, from=3-1, to=1-3]
	\arrow[from=3-1, to=3-3]
	\arrow["h", dashed, from=3-5, to=1-7]
	\arrow["{(u \circ r, id_Y)}"', from=3-5, to=3-7]
\end{tikzcd}\]

Applying the functor $\Pi_f$, we get a diagram

\[\begin{tikzcd}[ampersand replacement=\&]
	{\Pi_fX} \&\& {\Pi_fPX} \&\& {\Pi_fX} \& {\Pi_fPX} \& {\Pi_fPY} \\
	{\Pi_fY} \\
	{\Pi_fPY} \&\& {*} \&\& {\Pi_fY} \&\& {\Pi_fY \times_A' \Pi_fY}
	\arrow["{\Pi_f\iota_X}", from=1-1, to=1-3]
	\arrow["{\Pi_fu}"', from=1-1, to=2-1]
	\arrow[from=1-3, to=3-3]
	\arrow["{\Pi_f\iota_X}", from=1-5, to=1-6]
	\arrow["{\Pi_fu}"', from=1-5, to=3-5]
	\arrow["{\Pi_fPu}", from=1-6, to=1-7]
	\arrow[from=1-7, to=3-7]
	\arrow["{\Pi_f\iota_Y}"', from=2-1, to=3-1]
	\arrow["{\Pi_fPu}"', dashed, from=3-1, to=1-3]
	\arrow[from=3-1, to=3-3]
	\arrow["{\Pi_fh}", dashed, from=3-5, to=1-7]
	\arrow["{(\Pi_fu \circ \Pi_fr, id_{\Pi_fY})}"', from=3-5, to=3-7]
\end{tikzcd}\]
Now, the function extensionality structure we assumed for the $\Pi$-types in $\mathbf{C}$ provides us with a map $\mathbf{ext}_Y : \Pi_fPY \to P\Pi_fY$ over $\Pi_fY \times_A' \Pi_fY$, where $P\Pi_fY$ is again the path object induced by the identity type on $\Pi_f Y$. Hence, we have the following diagrams:

\[\begin{tikzcd}[ampersand replacement=\&]
	{\Pi_fX} \& {\Pi_fPX} \& {\Pi_fPY} \& {P\Pi_fY} \\
	\\
	{\Pi_fY} \&\& {\Pi_fY \times_A' \Pi_fY}
	\arrow["{\Pi_f\iota_X}", from=1-1, to=1-2]
	\arrow["{\Pi_fu}"', from=1-1, to=3-1]
	\arrow["{\Pi_fPu}", from=1-2, to=1-3]
	\arrow["{\mathbf{ext}_Y}", from=1-3, to=1-4]
	\arrow[from=1-3, to=3-3]
	\arrow[from=1-4, to=3-3]
	\arrow["{\Pi_fh}", from=3-1, to=1-3]
	\arrow["{(\Pi_fu \circ \Pi_fr, id_{\Pi_fY})}"', from=3-1, to=3-3]
\end{tikzcd}\]
where the top composite $\Pi_fX \to P\Pi_fY$ factors through $\iota_{\Pi_fX} : \Pi_fX  \to P\Pi_fX$ by the extra assumption on the extensionality structure $\mathbf{ext}$. This allows us to see that $\Pi_fu$ is a strong deformation retract, hence an anodyne map. This concludes the proof that $T \mathbf{C}$ is a $\pi$-tribe.
\end{proof}

\begin{proposition}
 \label{piR}
 Consider a tribe $\mathcal{T}$, and a homotopical inverse category $I$ where all arrows are weak equivalences. Then the category of Reedy fibrant diagrams $I \to \mathcal{T}$ can be endowed with a $\pi$-tribe structure $\mathcal{T}^I_\mathbf{R}$.
 Moreover: 
 \begin{itemize}
  \item If $\mathcal{T} \to \mathcal{S}$ is a morphism of $\pi$-tribes, then so is the induced morphism  $\mathcal{T}^I_\mathbf{R} \to \mathcal{S}^I_\mathbf{R}$.
  \item If $p : I \to J$ is a discrete opfibration (where we make on $J$ the same assumptions as on $I$), then precomposition by $p$ induces a morphism of $\pi$-tribes $\mathcal{T}^J_\mathbf{R} \to  \mathcal{T}^I_\mathbf{R}$.
 \end{itemize} 
\end{proposition}

\begin{proof}
 The tribe structure on the category $\mathcal{T}^I_\mathbf{R}$ is defined in \cite{ks2019internal} (see Definition 2.21 and Lemma 2.22). Therefore, we only need to check that this tribe is a $\pi$-tribe. By definition of the Reedy structure on both tribes (see Definition 2.21 in \cite{ks2019internal}) and categories with attributes (see Definition 3.22 in \cite{kapulkin2021homotopical}), the tribe $\mathcal{T}^I_\mathbf{R}$, is mapped by the functor $U \circ C : \mathbf{Trb} \to \mathbf{CompCat}_{\Sigma, id} \to \mathbf{CwA}_{\Sigma,id}$ (which is actually an equivalence of categories) to the category with attributes of homotopical (strict) Reedy types $(UC \mathcal{T})^I$. Here, $$U: \mathbf{CompCat}_{\Sigma, id} \to \mathbf{CwA}_{\Sigma,id}$$
 takes a (full) comprehension category $p : \mathbf{E} \to \mathbf{B}$ to a category with attributes by first replacing the Grothendieck fibration $p$ by an equivalent split one $p':\mathbf{E}' \to \mathbf{B}$, and then forgetting about the categorical structure of types above a given context $\Gamma$. Note that, modulo the Grothendieck construction, this last step boils down to post-composing with the object functor $\mathbf{Ob} : \mathbf{Cat} \to \mathbf{Set}$. There is also a functor  $$U: \mathbf{CwA}_{\Sigma,id} \to \mathbf{CompCat}_{\Sigma, id}$$ 
 that takes a category with attributes, thought of as a discrete comprehension category $p : \mathbf{E} \to \mathbf{B}$, to the full comprehension category $p' : \mathbf{E}' \to \mathbf{B}$ obtained by factoring the comprehension functor as a functor bijective on objects followed by a fully faithful functor:

\[\begin{tikzcd}[ampersand replacement=\&]
	{\mathbf{E}} \&\& {\mathbf{E}'} \&\& {\mathbf{B}^\rightarrow} \\
	\\
	\&\& {\mathbf{B}}
	\arrow[from=1-1, to=1-3]
	\arrow["p"', from=1-1, to=3-3]
	\arrow[from=1-3, to=1-5]
	\arrow["{p'}"', from=1-3, to=3-3]
	\arrow["{\mathbf{cod}}", from=1-5, to=3-3]
\end{tikzcd}\]

 Now, by \cref{compcat_pitribe}, $C \mathcal{T}$ supports extensional $\Pi$-types since $\mathcal{T}$ is  a $\pi$-tribe. It is equivalent to say that the category with attributes $UC \mathcal{T}$ supports these $\Pi$-types. Next, Proposition 5.13 of \cite{kapulkin2021homotopical} tells us that $(UC \mathcal{T})^I$ admits extensional $\Pi$-types. We also observe that $T \circ F \circ UC : \mathbf{Trb}_\pi \to \mathbf{Trb}_\pi$ is the identity functor. Therefore, we can conclude that $TF(UC \mathcal{T})^I \simeq TFUC \mathcal{T}^I_\mathbf{R} = \mathcal{T}^I_\mathbf{R}$  is a $\pi$-tribe, again by \cref{compcat_pitribe}.
 
 For the second point, since internal products in the tribe $\mathcal{T}$ correspond to $\Pi$-types in $UC \mathcal{T}$ (and likewise for $\mathcal{S}$), the fact that $\mathcal{T}^I_\mathbf{R} \to \mathcal{S}^I_\mathbf{R}$ is a $\pi$-closed morphism of tribes, provided that $\mathcal{T} \to \mathcal{S}$ is $\pi$-closed, and likewise for $p^* : \mathcal{T}^J_\mathbf{R} \to  \mathcal{T}^I_\mathbf{R}$, follows by preservation of the $\Pi$-types at the level of the corresponding categories with attributes. This is established in Proposition 5.14 of \cite{kapulkin2021homotopical}.

 \end{proof}

\begin{proof}[Proof of \cref{piP}]
 $P\mathcal{T}$ can equally be defined as the category of Reedy fibrant objects in the category of homotopical inverse diagrams on $\mathcal{T}$ of the following shape: $$\bullet \stackrel{\sim}{\leftarrow} \bullet \stackrel{\sim}{\rightarrow} \bullet$$ That the projections are $\pi$-closed follows directly from the construction of the dependent product for $P\mathcal{T}$ given by Proposition 5.13 in \cite{kapulkin2021homotopical}.
 \end{proof}

 \section{Some fibration categories of tribes}
 
 We begin this section by briefly explaining the problem that occurred in the first version of this document. With the notion of cube category considered previously, the category $\catname{scTrb}$ of semi-cubical tribes does not enjoy a fibration category structure in the expected canonical way, contrary to what we intended. Given a semi-cubical tribe $\mathcal{T}$, the factorization of the diagonal $\mathcal{T} \to \mathcal{T} \times \mathcal{T}$ through the projection $P\mathcal{T} \to \mathcal{T} \times \mathcal{T} $ given by
 \begin{align*}
  \iota_\mathcal{T} : & \mathcal{T} \to P\mathcal{T}\\
  & x \mapsto  (x \leftarrow x^{\square^1} \rightarrow x)
 \end{align*}
 is not a morphism in $\catname{scTrb}$, as it need not preserve the cotensors by finite semi-cubical sets. Precisely, the following diagram fails to commute:

\[\begin{tikzcd}[ampersand replacement=\&]
	{\mathcal{T}} \&\& {\mathcal{T}} \\
	\\
	{P\mathcal{T}} \&\& {P\mathcal{T}}
	\arrow["{(-)^K}", from=1-1, to=1-3]
	\arrow["{\iota_\mathcal{T}}"', from=1-1, to=3-1]
	\arrow["\neq"{description}, draw=none, from=1-1, to=3-3]
	\arrow["{\iota_\mathcal{T}}", from=1-3, to=3-3]
	\arrow["{(-)^K}"', from=3-1, to=3-3]
\end{tikzcd}\] 
 Indeed, given a semi-cubical set $K$ and an object $x$ of $\mathcal{T}$, the object $(x^{\square^1})^K \simeq x^{\square^1 \otimes K}$, obtained by applying the lower left composite to $x$ in the diagram above, need not be isomorphic to $(x^K)^{\square^1} \simeq x^{K \otimes \square^1}$, the object obtained by applying the upper right composite to $x$. The two objects are not isomorphic in general because the monoidal product on semi-cubical sets, derived from the monoidal structure of our original choice for the semi-cube category, is not symmetric.
 
 This suggests working with a category of cubes for which the monoidal product on semi-cubical sets is symmetric. This is the case as soon as the underlying category of cubes is symmetric moinoidal. Therefore, the sole category of semi-cubes we consider here is the free symmetric monoidal category $(\square_\sharp^s, \otimes, I^0)$ generated by two face maps $\delta_0, \delta_1 : I^0 \to I^1$ together with an involution of the interval (a reversal) $r : I^1 \to I^1$.
 
 As discussed in the previous section, it will be convenient to replace the category $\mathbf{Trb}$ of tribes and tribe morphisms between them, as well as the category $\mathbf{Trb_\pi}$ of $\pi$-tribes and $\pi$-closed morphisms of tribes between them, by various DK-equivalent subcategories, some of which happen to carry a fibration category structure. Regarding $\pi$-tribes, we will have to consider two types of morphisms: those that preserve the internal product up to isomorphism and those that preserve it only up to weak equivalence.
 
 The following categories, which are variations of $\catname{Trb}$, will be of interest to us.
 
 \begin{definition}
 We consider:
  \begin{itemize}
   \item The category $\catname{scTrb}$ is the category of semi-cubical tribes, as introduced in \cite[Definition 1.1]{cherradi2026semi}, with morphisms the semi-cubical functors that preserve the cotensors by finite semi-cubical presheaves and whose underlying functor is a morphism of tribe.
   \item The category $\catname{scTrb_\pi}$ is the category of semi-cubical tribes, that are moreover $\pi$-tribes, and, between them, the morphisms of semi-cubical tribes that are moreover $\pi$-closed (i.e, the pullback $\catname{scTrb} \times_{\catname{Trb}} \catname{Trb_\pi}$).   
   \item The category $\catname{Trb_{\pi,\sim}}$ is the category whose objects are the tribes that are equivalent to a $\pi$-tribe (in $\catname{Trb}$) and whose morphisms are the tribe morphisms $m : \mathcal{T} \to \mathcal{T}'$ such that $\mathbf{Ho}_\infty(m)$ is a locally cartesian closed $\infty$-functor.
   \item The category $\catname{scTrb_{\pi,\sim}}$ is the semi-cubical counterpart of $\catname{Trb_{\pi,\sim}}$, that is $\catname{scTrb_{\pi,\sim}} := \catname{scTrb} \times_{\catname{Trb}} \catname{Trb_{\pi,\sim}}$.
  \end{itemize}  
\end{definition}

\begin{remark}
 There is a notion of (semi-)simplicial $\pi$-tribe introduced by Joyal in \cite{joyal2017notes}, where the $\pi$-tribe structure plays well with the enrichment in that the universal property of the internal product is required to induce an isomorphism of enriched hom (which are (semi-)simplicial sets) rather than an isomorphism of homsets. A similar condition could be introduced for semi-cubical tribes. For our purpose, we will not have to consider the corresponding notion of semi-cubical $\pi$-tribe, but rather, as in the definition above, the notion of semi-cubical tribes that are also, and separately, $\pi$-tribes (in a non-enriched sense). Indeed, we will have no use for the enriched universal property of the internal product, so sticking to the latter notion will make things simpler.
\end{remark}

Unlike the category $\catname{scTrb}$, the category $\catname{scTrb_\pi}$ does not seem to carry a fibration category structure. This is because the morphism $\mathcal{T} \to P\mathcal{T}$ mapping $x$ to $x \leftarrow x^{\square_\sharp^1} \rightarrow x$, which is used to define the path-object for a semi-cubical tribe $\mathcal{T}$, is not $\pi$-closed in general (when assuming $\mathcal{T}$ to be a $\pi$-tribe, which also implies that $P\mathcal{T}$ is one). The following definition aims at "forcing" the existence of path-objects:

\begin{definition}
We define $\catname{scTrb_\pi^p}$ as the full subcategory of $\catname{scTrb_\pi}$ spanned by those objects $\mathcal{T}$ such that the morphism $\iota_{\mathcal{T}} : \mathcal{T} \to P\mathcal{T}$ supplied by the semi-cubical structure is a $\pi$-closed morphism of tribes.
\end{definition}

\begin{remark}
 Although the previous definition may appear ad hoc, it is reminiscent of the principles stipulated by higher observational type theory (see \cite{nlab:higher_observational_type_theory}), introduced by Altenkirch, Kaposi and Shulman. Indeed, in this framework, the identity type associated with a dependent function type $\Pi_{x :A} B(x)$ is defined by:
 
 $$ f =_{\Pi_{x :A} B(x)} g \equiv \Pi_{a : A} \Pi_{b : A} \Pi_{p : a =_A b} (f(a) =_B^p g(b))$$
 
 This matches the condition defining $\catname{scTrb_\pi^p}$, since, given two fibrations $B \to A$ and $A \to \Gamma$ in a $\pi$-tribe $\mathcal{T}$, computing the dependent product after taking the identity type (i.e, the forming the cotensor by $\square_\sharp^1$, as in $\iota_\mathcal{T}$) would yield the previous expression.
\end{remark}

\begin{remark}
 In type theory, constructing models of MLTT where this condition holds can be done within the framework of parametricity. The connection between parametricity and (semi)-cubical structures has been observed in the literature (see \cite{moeneclaey2021parametricity}), and motivates the shift from semi-simplicial to semi-cubical structures relative to \cite{ks2019internal}. Indeed, as we will establish later here, the "recipe" to connect tribes to semi-cubical tribes, via the so-called frame construction, yields "parametric" tribes, and hence lands directly in $\catname{scTrb_\pi^p}$. In the language of \cite{moeneclaey2021parametricity}, the frame construction is in fact the "free parametric model" functor, as discussed in Example 35 there.
 \end{remark}

 To summarize, we have the following commutative diagram,  
\[\begin{tikzcd}[ampersand replacement=\&]
	\textcolor{rgb,255:red,204;green,51;blue,51}{\catname{scTrb_\pi^p}} \&\& {\catname{scTrb_\pi}} \&\& {\mathbf{Trb_\pi}} \\
	\\
	\&\& \textcolor{rgb,255:red,204;green,51;blue,51}{\catname{scTrb_{\pi,\sim}}} \&\& {\mathbf{Trb_{\pi,\sim}}} \\
	\\
	\&\& \textcolor{rgb,255:red,204;green,51;blue,51}{\mathbf{scTrb}} \&\& {\mathbf{Trb}}
	\arrow["\sim"{description}, from=1-1, to=1-3]
	\arrow["\sim"{description}, from=1-3, to=1-5]
	\arrow["\sim"{description}, from=1-3, to=3-3]
	\arrow["\sim"{description}, from=1-5, to=3-5]
	\arrow["\sim"{description}, from=3-3, to=3-5]
	\arrow[from=3-3, to=5-3]
	\arrow[from=3-5, to=5-5]
	\arrow["\sim"{description}, from=5-3, to=5-5]
\end{tikzcd}\]
where the categories of tribes in red can be equipped with a fibration category structure, and where the indicated morphisms are DK-equivalences, as we will show later in this document.

 \begin{remark}
  Two tribes are equivalent in $\mathbf{Trb}$ if they are connected by a zig-zag of weak equivalences, namely morphisms of tribes which are DK-equivalences, or equivalently, by Cisinski's theorem, which induce an equivalence of categories at the level of the homotopy categories.  
  Importantly, since a $\pi$-tribe $\mathcal{T}'$ induces a locally cartesian closed $\mathbf{Ho}_\infty(\mathcal{T}')$ quasicategory, and since as a weak equivalence between tribes $f : \mathcal{T} \to \mathcal{T}'$ induces an equivalence of quasicategories $\mathbf{Ho}_\infty(f) : \mathbf{Ho}_\infty(\mathcal{T}) \to \mathbf{Ho}_\infty(\mathcal{T}')$, it follows that a tribe $\mathcal{T}$ in $\catname{scTrb_{\pi,\sim}}$ also yields a locally cartesian closed quasicategory $\mathbf{Ho}_\infty(\mathcal{T})$.  
 \end{remark}

 We recall the following standard fact about locally cartesian closed categories. The transposed statement to the setting of quasicategories is also true.
 
 \begin{proposition}
  A finitely complete category $\mathbf{C}$ is locally cartesian closed if and only if all the slice categories $\mathbf{C}_{/x}$ are cartesian closed, for $x$ an object of $\mathbf{C}$. Moreover, a functor $F:  \mathbf{C} \to \mathbf{D}$ between locally cartesian closed categories is a locally cartesian closed functor if and only if all induced functors $F_x:  \mathbf{C}_{/x} \to \mathbf{D}_{/Fx}$ are cartesian closed.
 \end{proposition}

 \begin{remark}
 \label{piclosed}
  Given two $\pi$-tribes and a morphism of tribes $m : \mathcal{T} \to \mathcal{T}'$, we have, for every object $x$ of $\mathcal{T}$ (also seen as an object of $\mathbf{Ho}_\infty(\mathcal{T})$), two canonical transformations

\begin{equation}
 \label{transpi}
  \begin{tikzcd}
	{\mathcal{T}} && {\mathcal{T}'} && {\mathbf{Ho}_\infty(\mathcal{T})} && {\mathbf{Ho}_\infty(\mathcal{T}')} \\
	\\
	{\mathcal{T}} && {\mathcal{T}'} && {\mathbf{Ho}_\infty(\mathcal{T})} && {\mathbf{Ho}_\infty(\mathcal{T}')}
	\arrow["{(-)^x}"', from=1-1, to=3-1]
	\arrow["m"', from=3-1, to=3-3]
	\arrow["{(-)^{ m(x)}}", from=1-3, to=3-3]
	\arrow["m", from=1-1, to=1-3]
	\arrow[shorten <=17pt, shorten >=17pt, Rightarrow, from=3-1, to=1-3]
	\arrow["{(-)^{x}}"', from=1-5, to=3-5]
	\arrow["{\mathbf{Ho}_\infty(m)}"', from=3-5, to=3-7]
	\arrow["{(-)^{ m(x)}}", from=1-7, to=3-7]
	\arrow["{\mathbf{Ho}_\infty(m)}", from=1-5, to=1-7]
	\arrow[shorten <=19pt, shorten >=19pt, Rightarrow, from=3-5, to=1-7]
\end{tikzcd}
\end{equation}
which arise from the universal property of the exponential.

The homotopy category $\mathbf{Ho}(\mathcal{T})$ is also cartesian closed as the homotopy category of a $\pi$-tribe, and the previous two transformations both induce the following canonical transformation:

\begin{equation}
\label{derivedt}
\begin{tikzcd}
	{\mathbf{Ho}(\mathcal{T})} && {\mathbf{Ho}(\mathcal{T}')} \\
	\\
	{\mathbf{Ho}(\mathcal{T})} && {\mathbf{Ho}(\mathcal{T}')}
	\arrow["{(-)^{x}}"', from=1-1, to=3-1]
	\arrow["{\mathbf{Ho}(m)}"', from=3-1, to=3-3]
	\arrow["{(-)^{ m(x)}}", from=1-3, to=3-3]
	\arrow["{\mathbf{Ho}(m)}", from=1-1, to=1-3]
	\arrow[shorten <=18pt, shorten >=18pt, Rightarrow, from=3-1, to=1-3]
\end{tikzcd}
\end{equation}

More generally, by definition, the induced morphism $m_A : \mathcal{T}(A) \to \mathcal{T}'(mA)$ fits in similar diagrams, where $A$ is an object of $\mathcal{T}$, and $\mathcal{T}(A)$ is the full subcategory of the slice $\mathcal{T}_{/A}$ spanned by the fibrations (it is referred to as the \textit{fibrant} slice over $A$).
Observe that the internal product of a fibration along a fibration, in the sense of Definition 2.4.1 of \cite{joyal2017notes}, can be expressed from exponentials in a fibrant slice. In particular, we see that a morphism $m$ in $\catname{scTrb}$ is a morphism in $\catname{scTrb_{\pi,\sim}}$ if and only if, for every object $A$ of $\mathcal{T}$, $\mathbf{Ho}_\infty(m_A)$ is cartesian closed (in the suitable sense for quasicategories). This is, in turn, equivalent to requiring that for all $x$ in $\mathcal{T}(A)$, the derived transformation in \eqref{derivedt} is a natural isomorphism.
Equivalently, the property holds precisely when the comparison arrow $m_A(y^x) \to m_A(y)^{m_A(x)}$ is a weak equivalence for all objects $A$ of $\mathcal{T}$, and all objects  $x$ and $y$ of the fibrant slice $\mathcal{T}(A)$. 
This provides a practical criterion to establish that a morphism of tribes $m$ between two $\pi$-tribes is in $\catname{scTrb_{\pi,\sim}}$.
Finally, note that this is equivalent to any of the two transformations in \eqref{transpi} being invertible (hence both of them).
 \end{remark}
 
 \begin{remark}
  \label{eqpitribe}  
  A tribe $\mathcal{T}$ is equivalent, in $\mathbf{Trb}$, to a $\pi$-tribe (that is connected by a zig-zag of morphisms of tribes which are DK-equivalences)  if and only if $\mathbf{Ho}_\infty(\mathcal{T})$ is a locally cartesian closed quasicategory. The inverse implication is trivial. For the direct one, the result can be deduced from \cite[Theorem 2.4]{cherradi2022interpreting}, which provides a $\pi$-tribe $\mathcal{T}'$ such that $\mathbf{Ho}_\infty(\mathcal{T}') \simeq  \mathbf{Ho}_\infty(\mathcal{T})$, together with the fact, proved in  \cite{ks2019internal}, that $\catname{Trb} \to \catname{QCat}_{lex}$ is a DK-equivalence so that the (zig-zag of)  equivalence between $\mathbf{Ho}_\infty(\mathcal{T}')$ and $\mathbf{Ho}_\infty(\mathcal{T})$  implies the existence of a zig-zag of equivalences between $\mathcal{T}$ and $\mathcal{T}'$.
 \end{remark}

 We take the notion of fibrations and weak equivalences between tribes as in \cite{ks2019internal}, namely:
 
 \begin{definition}
 A weak equivalence between tribes is a morphism of tribes $F : \mathcal{T} \to \mathcal{T}'$ that is, moreover, a DK-equivalence. Equivalently, by Cisinski's characterization of DK-equivalence between fibration categories in \cite{cisinski2010invariance}, it is a morphism of tribe inducing an equivalence of categories $\mathbf{Ho}(F) : \mathbf{Ho}(\mathcal{T}) \to \mathbf{Ho}(\mathcal{T}')$.
 \end{definition}
 
 \begin{definition}
  A morphism of tribes $F : \mathcal{T} \to \mathcal{T}'$ is a fibration when:
  \begin{itemize}
   \item[$a.$] $F$ is an isofibration.
   \item[$b_1.$] any factorization of $Fk$

\[\begin{tikzcd}
	Fx && Fy \\
	& {z'}
	\arrow["Fk", from=1-1, to=1-3]
	\arrow[from=1-1, to=2-2]
	\arrow[from=2-2, to=1-3]
\end{tikzcd}\]
   as a weak equivalence followed by a fibration in $\mathcal{T}'$, where $k : x \to y$ is a morphism in $\mathcal{T}$, lifts to a factorization as a weak equivalence followed by a  fibration in $\mathcal{T}$.
   
   \item[$b_2.$] similarly, any factorization of $Fk$ as an anodyne map followed by a fibration lifts to a factorization as an anodyne map followed by a fibration in $\mathcal{T}$.
   \item[$c.$] any  \textit{pseudo-factorization} of $Fk$ in $\mathcal{T}'$,

\[\begin{tikzcd}
	Fx && Fy \\
	\\
	{z_0} && {z_1}
	\arrow["Fk", from=1-1, to=1-3]
	\arrow["\sim"{description}, two heads, from=3-1, to=1-1]
	\arrow["\sim"{description}, from=3-1, to=3-3]
	\arrow[two heads, from=3-3, to=1-3]
\end{tikzcd}\]
where the indicated maps are weak equivalences, lifts to one in $\mathcal{T}'$.
\item[$d_1.$] For a square as on the left below,
\[\begin{tikzcd}
	a && x && Fa && Fx \\
	\\
	b && y && Fb && Fy
	\arrow[from=1-5, to=1-7]
	\arrow["\sim"{description}, tail, from=1-5, to=3-5]
	\arrow[two heads, from=1-7, to=3-7]
	\arrow[from=3-5, to=3-7]
	\arrow["{h'}"', dashed, from=3-5, to=1-7]
	\arrow[from=3-1, to=3-3]
	\arrow[two heads, from=1-3, to=3-3]
	\arrow["\sim"{description}, tail, from=1-1, to=3-1]
	\arrow[from=1-1, to=1-3]
\end{tikzcd}\]
any lift $h'$ of its image through $F$ can be lifted to a lift of the square in $\mathcal{T}$.
\item[$d_2.$] For a "cofibrancy" lifting problem as on the left below,
\[\begin{tikzcd}
	&& x &&&& Fx \\
	\\
	b && y && Fb && Fy
	\arrow["\sim"{description}, two heads, from=1-7, to=3-7]
	\arrow[from=3-5, to=3-7]
	\arrow["{h'}"', dashed, from=3-5, to=1-7]
	\arrow[from=3-1, to=3-3]
	\arrow["\sim"{description}, two heads, from=1-3, to=3-3]
\end{tikzcd}\]
   any solution $h'$ of its image by $F$ can be lifted to a solution of the lifting problem in $\mathcal{T}$.
   
  \end{itemize}
  
  \begin{proposition}
   Fibrations are closed by pullback (computed in $\catname{Cat}$). 
  \end{proposition}
  
  \begin{proof}
   This is proved in Lemma 4.7 of \cite{ks2019internal} in the setting of semi-cubical tribes, but this assumption plays no role in the proof.
  \end{proof}

  \begin{example}
  \label{projP}
   Given a tribe $\mathcal{T}$, the  projection $P\mathcal{T} \to \mathcal{T}\times \mathcal{T}$ is fibration. This is proved in Lemma 4.5 of \cite{ks2019internal} in the case where $\mathcal{T}$ is a semi-cubical tribe, but the proof does not rely on this fact.
  \end{example}

 \end{definition}

   \begin{remark}
   Observe that the category $\catname{scTrb_\pi}$ also admits pullbacks along isofibrations (which are then computed in the category $\mathbf{Cat}_{\square_\sharp}$ of semi-cubically enriched categories). This is because, in a pullback square along an isofibration as follows

\[\begin{tikzcd}
	{\mathcal{T}} && {\mathcal{T}_2} \\
	\\
	{\mathcal{T}_1} && {\mathcal{T}_0}
	\arrow[from=3-1, to=3-3]
	\arrow[two heads, from=1-3, to=3-3]
	\arrow[from=1-1, to=3-1]
	\arrow[from=1-1, to=1-3]
	\arrow["\lrcorner"{anchor=center, pos=0.125}, draw=none, from=1-1, to=3-3]
\end{tikzcd}\]
   the category $\mathcal{T}$ with the component-wise notion of fibration admits internal products which are computed component-wise. It also follows that the universal property of this diagram makes it a pullback in $\catname{scTrb_\pi}$, because any cone of $\pi$-closed morphisms on the previous cospan will induce a $\pi$-closed mediating map into the pullback $\mathcal{T}$. 
  \end{remark}

  \begin{lemma}
  \label{Pspi}
   The endofunctor $P : \mathbf{Trb} \to \mathbf{Trb}$ restricts to an endofunctor of $\mathbf{scTrb}_\pi$.
  \end{lemma}
  \begin{proof}
   By \cref{piP}, if $\mathcal{T}$ is a tribe in $\mathbf{scTrb}_\pi$, then $P\mathcal{T}$ is a $\pi$-tribe. Moreover, if $F : \mathcal{T} \to \mathcal{S}$ is a $\pi$-closed morphism of $\pi$-tribes, it follows from \cref{piR} (using the characterization of $P\mathcal{T}$ as the category of Reedy fibrant homotopical span diagrams in $\mathcal{T}$) that $PF :  P\mathcal{T} \to P\mathcal{S}$ is also $\pi$-closed. 
   
   Also, by Lemma 2.22 in \cite{ks2019internal}, $P\mathcal{T}$ is a tribe. Moreover, the cotensors are defined pointwise, as we can see from the following natural isomorphisms, for $K$ a finite semi-cubical set, and for $(X,Y)$ a pair of objects in $P\mathcal{T}$:
   
   \begin{align*}
   Hom(K, Hom_{P\mathcal{T}}(X,Y)) & = Hom(K, \int_{c \in \mathbf{Sp}_w} Hom_\mathcal{T}(X(c),Y(c)))\\
   & \simeq \int_{c \in \mathbf{Sp}_w} Hom(K,Hom_\mathcal{T}(X(c),Y(c)))\\
   & \simeq \int_{c \in \mathbf{Sp}_w} Hom_\mathcal{T}(X(c),Y(c)^K)\\
   & = Hom_{P\mathcal{T}}(X,Y^K)
   \end{align*}
   
   Because the cotensors are pointwise, any semi-cubical functor $F : \mathcal{T} \to \mathcal{S}$ induces a functor $PF :  P\mathcal{T} \to P\mathcal{S}$ that is semi-cubical.
   
   This proves that the mapping $\mathcal{T} \mapsto P\mathcal{T}$ restricts to an endofunctor of $\mathbf{scTrb}_\pi$.   
  \end{proof}
  
  \begin{lemma}
  \label{iotapip}
   If $\mathcal{T}$ is a tribe in $\mathbf{scTrb}_\pi^p$, then so is $P\mathcal{T}$.
  \end{lemma}
  \begin{proof}
  Under the assumption on $\mathcal{T}$, $P\mathcal{T}$ is a tribe in $\mathbf{scTrb}_\pi$, by \cref{Pspi}. We need to check that the morphism of tribes $\iota_{P\mathcal{T}} : P\mathcal{T} \to PP\mathcal{T}$ is $\pi$-closed. But $\iota_{P\mathcal{T}}$ coincides with $P\iota_{\mathcal{T}}$ (modulo composing with the automorphism of $PP\mathcal{T}$ that swaps the roles of the two spans), as, by definition of the cotensors in $P\mathcal{T}$ and $PP\mathcal{T}$ (which are pointwise), both functors map the diagram below left, thought of as an object of $P\mathcal{T}$, to the diagram below right, thought of as an object of $PP\mathcal{T}$.

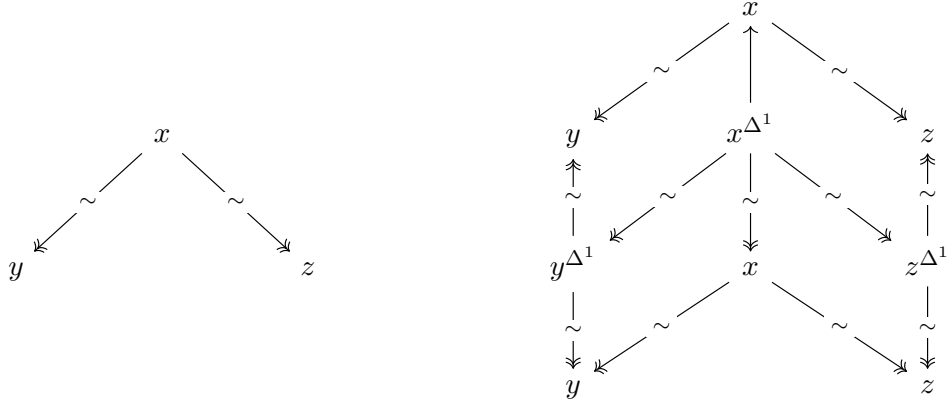
\[\begin{tikzcd}[ampersand replacement=\&,sep=large]
	\&\&\&\&\& x \\
	\& x \&\&\& y \& {x^{\Delta^1}} \& z \\
	y \&\& z \&\& {y^{\Delta^1}} \& x \& {z^{\Delta^1}} \\
	\&\&\&\& y \&\& z
	\arrow["\sim"{description}, two heads, from=1-6, to=2-5]
	\arrow["\sim"{description}, two heads, from=1-6, to=2-7]
	\arrow["\sim"{description}, two heads, from=2-2, to=3-1]
	\arrow["\sim"{description}, two heads, from=2-2, to=3-3]
	\arrow[from=2-6, to=1-6]
	\arrow["\sim"{description}, two heads, from=2-6, to=3-5]
	\arrow["\sim"{description}, two heads, from=2-6, to=3-6]
	\arrow["\sim"{description}, two heads, from=2-6, to=3-7]
	\arrow["\sim"{description}, two heads, from=3-5, to=2-5]
	\arrow["\sim"{description}, two heads, from=3-5, to=4-5]
	\arrow["\sim"{description}, two heads, from=3-6, to=4-5]
	\arrow["\sim"{description}, two heads, from=3-6, to=4-7]
	\arrow["\sim"{description}, two heads, from=3-7, to=2-7]
	\arrow["\sim"{description}, two heads, from=3-7, to=4-7]
\end{tikzcd}\]
  Note that this representation does not fully account for the Reedy fibrancy condition on the diagrams, but isomorphisms between two Reedy fibrant homotopical diagrams of shape $\mathbf{Sp}_w \times \mathbf{Sp}_w$ (where $\mathbf{Sp}_w$ is the homotopical span category) boils down to isomorphisms between these two diagrams as objects of the category $\mathcal{T}^{\mathbf{Sp}_w \times \mathbf{Sp}_w}$ (that is, without Reedy fibrancy criterion, or even the requirement that the diagrams are homotopical).
  
  The action on morphisms is similar. By \cref{piR}, the fact that $\iota_{\mathcal{T}}$ is $\pi$-closed implies the same property for $P\iota_{\mathcal{T}}$, so we can conclude.
  \end{proof}
  
\begin{proposition}
 $\catname{scTrb_\pi^p}$ can be endowed with the structure of a fibration category.
\end{proposition}

\begin{proof}

 We define fibrations (resp. a weak equivalences) in $\catname{scTrb_\pi^p}$ to be fibrations (resp. a weak equivalences) as a morphisms of $\catname{scTrb}$. To prove that  $\catname{scTrb_\pi^p}$ inherits a fibration category structure from $\catname{scTrb}$, we just need to prove that it is closed under pullback along fibrations, and that, for all objects $\mathcal{T}$, there exists a factorization of the diagonal functor $\mathcal{T} \to \mathcal{T} \times \mathcal{T}  $ that lies inside $\catname{scTrb_\pi^p}$. The second point holds because $P\mathcal{T}$ induces the expected factorization, since we know that $P\mathcal{T} \in \catname{scTrb_\pi^p}$ from \cref{iotapip}. For the first one, consider a pullback in $\catname{scTrb_\pi}$
\[\begin{tikzcd}[ampersand replacement=\&]
	{\mathcal{T}} \&\&\& {\mathcal{T}_1} \\
	\\
	{\mathcal{T}_2} \&\&\& {\mathcal{T}_0}
	\arrow[from=1-1, to=1-4]
	\arrow[two heads, from=1-1, to=3-1]
	\arrow["\lrcorner"{anchor=center, pos=0.125}, draw=none, from=1-1, to=3-4]
	\arrow[two heads, from=1-4, to=3-4]
	\arrow[from=3-1, to=3-4]
\end{tikzcd}\]
where $\mathcal{T}_0$, $\mathcal{T}_1$ and $\mathcal{T}_2$ lie in $\catname{scTrb_\pi^p}$.
 Because the internal products and the cotensors in the pullback $\mathcal{T}$ are defined component-wise, it follows from $\iota_{\mathcal{T}_i} : \mathcal{T}_i \to P\mathcal{T}_i$ being $\pi$-closed (for $i=0,1,2$) that the morphism $\iota_{\mathcal{T}} : \mathcal{T} \to P\mathcal{T}$ induced by the cotensoring is also $\pi$-closed.
 This proves that $\catname{scTrb_\pi^p}$ inherits the structure of a fibration category.
 \end{proof}

\begin{proposition}
 $\catname{scTrb_{\pi,\sim}}$ can be endowed with the structure of a fibration category.
\end{proposition}

\begin{proof}
 This is just a sub-fibration category of $\catname{scTrb}$, that is, it is closed under pullbacks along fibrations and contains some path-object factorization as in the definition of $\catname{scTrb}$. These two properties readily imply that, with the same notion of fibrations and weak equivalences, $\catname{scTrb_{\pi,\sim}}$ is a fibration category.
 
 The second point is clear, essentially because any weak equivalence $u : \mathcal{T} \to \mathcal{T}'$ in $\catname{scTrb}$, where $\mathcal{T}$ (hence also $\mathcal{T}'$) is equivalent to a $\pi$-tribe yields, in particular, a locally cartesian closed $\infty$-functor $\mathbf{Ho}_\infty(u)$.
 To see the first point, consider a pullback diagram along a fibration in $\catname{scTrb}$ as follows.

\[\begin{tikzcd}
	{\mathcal{T}} && {\mathcal{T}_1} \\
	\\
	{\mathcal{T}_2} && {\mathcal{T}_0}
	\arrow[from=3-1, to=3-3]
	\arrow[from=1-3, to=3-3]
	\arrow[from=1-1, to=3-1]
	\arrow[from=1-1, to=1-3]
	\arrow["\lrcorner"{anchor=center, pos=0.125}, draw=none, from=1-1, to=3-3]
\end{tikzcd}\]
 
 This is a homotopy pullback, which means that the induced square

\[\begin{tikzcd}
	{\mathbf{Ho}_\infty(\mathcal{T})} && {\mathbf{Ho}_\infty(\mathcal{T}_1)} \\
	\\
	{\mathbf{Ho}_\infty(\mathcal{T}_2)} && {\mathbf{Ho}_\infty(\mathcal{T}_0)}
	\arrow[from=3-1, to=3-3]
	\arrow[from=1-3, to=3-3]
	\arrow[from=1-1, to=3-1]
	\arrow[from=1-1, to=1-3]
\end{tikzcd}\]
 
 exhibits $\mathbf{Ho}_\infty(\mathcal{T})$ as a pullback of locally cartesian closed quasicategories in the (large) quasicategory $\mathcal{Q}Cat$ of (small) quasicategories. Thus, $\mathbf{Ho}_\infty(\mathcal{T})$ is locally cartesian closed, which implies, by Remark \ref{eqpitribe}, that $\mathcal{T}$ is equivalent to a $\pi$-tribe. The very same argument, using Remark \ref{piclosed}, gives a proof that the pullback square above also defines a pullback in $\catname{scTrb_{\pi,\sim}}$. \end{proof}

\section{The rigidification tool}

The following construction, despite its simplicity, is the result at the core of the present article. 

The main idea can be broadly stated as follows: when structure in a tribe is preserved only up to weak equivalence by a functor $F : \mathcal{T} \to  \mathcal{S}$, the data provided by the (canonical) weak equivalences "lift" $F$, in a sense, through the canonical path object $P\mathcal{S}$. This allows one to "factor" $F$ as a span of functors that preserve the structure up to isomorphism. The precise construction can be thought of as an instance of Artin gluing. It can also be seen as an instance of oplax limits in the sense of Definition 12.3 in \cite{shulman2015} (or rather a variation of this definition where $I$ is an inverse homotopical category, and where we ask the morphism $A_\alpha : Ax \to  \alpha^*(A_y)$ to be a weak equivalence whenever $\alpha$ is a weak equivalence), namely, we consider the following diagram $D : \mathbf{Sp}_w^{op} \to \mathbf{Cat}$
\[\begin{tikzcd}[ampersand replacement=\&]
	\& {\mathcal{S}} \\
	{\mathcal{T}} \&\& {\mathcal{S}}
	\arrow["F", from=2-1, to=1-2]
	\arrow["{id_{\mathcal{S}}}"', from=2-3, to=1-2]
\end{tikzcd}\]
where $\mathbf{Sp}_\simeq$ is the span-shaped homotopical category where all arrows are weak equivalences. The category of interest would then be $\llbracket \mathbf{Sp}_w, D \rrbracket_\mathbf{f}$ with the notations of \cite{shulman2015}.

In preparation for the proof of the main result of this section (Lemma \ref{strf}), we start by establishing the following statement, dealing only with exponentials:
\begin{lemma}
\label{strf0}

	Suppose $F : \mathcal{T} \to \mathcal{S}$ is a morphism of tribes between two $\pi$-tribes in $\mathbf{scTrb}$, and assume that $\catname{\mathbf{Ho}_\infty}(F)$ is a cartesian closed $\infty$-functor. Then there exists a span $\mathcal{T}' \to \mathcal{T} \times \mathcal{S}$ of semi-cubical tribes such that $\mathcal{T}'$ admits exponentials, $\mathcal{T}' \to \mathcal{T}$ is an exponential preserving weak equivalence in $\mathbf{Trb}$, and $\mathcal{T}' \to \mathcal{S}$ is an exponential-preserving functor, fitting in a commutative diagram:
 
\[\begin{tikzcd}
	& {\mathcal{T}} \\
	{\mathcal{T}} && {\mathcal{S}} \\
	& {\mathcal{T}'}
	\arrow["\sim"{description}, from=3-2, to=2-1]
	\arrow[from=3-2, to=2-3]
	\arrow["{id_{\mathcal{T}}}"', from=1-2, to=2-1]
	\arrow["F", from=1-2, to=2-3]
	\arrow["m"', from=1-2, to=3-2]
\end{tikzcd}\]
 Here, $m$ is necessarily a weak equivalence; furthermore $\mathcal{T}' \to \mathcal{S}$ is a fibration, and $\mathcal{T}' \to \mathcal{T}$ is a trivial fibration.
\end{lemma}

\begin{proof}
The projection functor $P\mathcal{S} \to \mathcal{S} \times \mathcal{S}$ is a fibration, as mentioned in Example \ref{projP}. Moreover, by construction of the exponentials in $P\mathcal{S}$, the two projections $\partial_0,  \partial_1 : P\mathcal{S} \to \mathcal{S}$  preserve them.  We construct the expected span by forming the following pullback (computed in $\mathbf{Cat}$):
\begin{equation}
\label{pbr}
 \begin{tikzcd}
	{\mathcal{T}'} && {P\mathcal{S}} \\
	\\
	{\mathcal{T} \times \mathcal{S}} && {\mathcal{S} \times \mathcal{S}}
	\arrow["{<\partial_0,\partial_1>}", two heads, from=1-3, to=3-3]
	\arrow["{F \times id_{\mathcal{S}}}"', from=3-1, to=3-3]
	\arrow[two heads, from=1-1, to=3-1]
	\arrow["u", from=1-1, to=1-3]
	\arrow["\lrcorner"{anchor=center, pos=0.125}, draw=none, from=1-1, to=3-3]
\end{tikzcd}
\end{equation}

 We will show that $\mathcal{T}'$ has exponentials, and that the projections $\mathcal{T}' \to \mathcal{T}$ and $\mathcal{T}' \to \mathcal{S}$ preserve those.
Let $A = (a,x, v \to Fa \times x)$ and $B = (b,y, w \to Fb \times y)$  be two objects of $\mathcal{T}'$, where the elements of tuples correspond in order to the components in $\mathcal{T}$, $\mathcal{S}$ and $P\mathcal{S}$ (thinking of the spans that are objects of $P\mathcal{S}$ as a single arrow into a product). Consider the exponential $uA^{uB}$ in $P\mathcal{S}$. By definition of $u$, one has that $ \partial_0 \circ u = F \circ \partial'_0$, and thus that $\partial_0 \circ uA = Fa$, and $\partial_0 \circ uB = Fb$. It follows that the exponential $uA^{uB}$ is transported to $Fa^{Fb}$ by $\partial_0 : P\mathcal{S} \to \mathcal{S}$. Similarly, $uA^{uB}$ is mapped to $x^y$ by $\partial_1 : P\mathcal{S} \to \mathcal{S}$. It follows that $uA^{uB}$ is of the form $(z \to F(a)^{F(b)} \times x^y)$.
 
  Now, we claim that the exponential $A^B$ in $\mathcal{T}'$ can be defined as the tuple $S = (a^b,x^y,s \to F(a^b) \times x^y)$, where $s \to F(a^b) \times x^y$ is constructed by pullback in $\mathcal{S}$ as follows:

\begin{equation}
\label{pb1}
\begin{tikzcd}
	s && z \\
	\\
	{F(a^b) \times x^y} && {F(a)^{F(b)} \times x^y}
	\arrow[from=3-1, to=3-3]
	\arrow[two heads, from=1-3, to=3-3]
	\arrow[two heads, from=1-1, to=3-1]
	\arrow[from=1-1, to=1-3]
	\arrow["\lrcorner"{anchor=center, pos=0.125}, draw=none, from=1-1, to=3-3]
\end{tikzcd}
\end{equation}
We show that the span $s \to F(a^b) \times x^y$ defines an object of $\mathcal{T}'$, namely, that it consists of a span of trivial fibrations. This is established by observing that, by hypothesis on $F$, and using Remark \ref{piclosed}, the canonical map $F(a^b) \to Fa^{Fb}$ is a weak equivalence in $\mathcal{S}$. From this, it follows that the bottom map $F(a^b) \times x^y \to F(a)^{F(b)} \times x^y$ is also a weak equivalence, hence also the top map $s \to z$, since weak equivalences in a tribe are stable by pullbacks along fibrations. We now have two commutative squares as follows:
\[\begin{tikzcd}
	s && z && s && z \\
	\\
	{F(a^b)} && {F(a)^{F(b)}} && {x^y} && {x^y}
	\arrow["\sim"{description}, from=3-1, to=3-3]
	\arrow["\sim"{description}, two heads, from=1-3, to=3-3]
	\arrow[two heads, from=1-1, to=3-1]
	\arrow["\sim"{description}, from=1-1, to=1-3]
	\arrow["{id_{x^y}}"', from=3-5, to=3-7]
	\arrow["\sim"{description}, two heads, from=1-7, to=3-7]
	\arrow[two heads, from=1-5, to=3-5]
	\arrow["\sim"{description}, from=1-5, to=1-7]
\end{tikzcd}\]
Finally, we conclude, by the 2-out-of-3 property for weak equivalences, that, in the two diagrams above, both $s \to F(a^b)$ and $s \to x^y$ are weak equivalences. This establishes that $S$ is an object of $\mathcal{T}'$. We explain now why it defines the exponential object $A^B$ in $\mathcal{T}'$.

Note that $S$ comes with a map $\epsilon': S \times B \to A$ obtained by composing the evaluation map in $P\mathcal{S}$

\[\begin{tikzcd}
	{z \times w} && v \\
	\\
	{F(a)^{F(b)}\times x^y \times F(b) \times y} && {F(a) \times x}
	\arrow[two heads, from=1-1, to=3-1]
	\arrow[from=3-1, to=3-3]
	\arrow[from=1-1, to=1-3]
	\arrow[from=1-3, to=3-3]
\end{tikzcd}\]
with the square

\[\begin{tikzcd}
	{s \times w} && {z \times w} \\
	\\
	{F(a^b) \times F(b) \times x^y \times y} && {F(a)^{F(b)} \times F(b)\times x^y \times y}
	\arrow[from=3-1, to=3-3]
	\arrow[two heads, from=1-3, to=3-3]
	\arrow[two heads, from=1-1, to=3-1]
	\arrow[from=1-1, to=1-3]
	\arrow["\lrcorner"{anchor=center, pos=0.125}, draw=none, from=1-1, to=3-3]
\end{tikzcd}\]
We prove that $\epsilon'$ satisfies the property for the evaluation map of an exponential object $A^B$. Consider an object $C = (c,y', w' \to Fc \times y')$ of $\mathcal{T}'$ with a morphism $k : C \times B \to A$.
By definition, the map $u(k)$ factors through the evaluation morphism $\epsilon : u(A)^{u(B)} \times u(B) \to u(A)$, yielding a uniquely defined morphism $\lambda(k) : u(C) \to u(A)^{u(B)}$.
In particular we have a square

\[\begin{tikzcd}
	{w'} && z \\
	\\
	{F(c) \times y'} && {F(a)^{F(b)} \times x^y}
	\arrow[from=3-1, to=3-3]
	\arrow[two heads, from=1-3, to=3-3]
	\arrow[two heads, from=1-1, to=3-1]
	\arrow[from=1-1, to=1-3]
\end{tikzcd}\]
where the bottom map factors uniquely through $(F(a^b) \times F(b) )\times x^y$ by the defining property of the exponential in $\mathcal{T}$. By the universal property of the pullback \eqref{pb1}, we get the expected unique factorization in $P\mathcal{S}$, which yields a uniquely defined factorization of $k$ through the evaluation $\epsilon': S \times B \to A$. This proves that $\mathcal{T}'$ admits exponentials. We also observe that, if $A \to A'$ is a fibration in $\mathcal{T}'$, then the induced map $A^B \to A'^B$ will also be a fibration, since its $P\mathcal{S}$ component is obtained via pullback from the map $uA^{uB} \to uA'^{uB}$, which is a fibration by virtue of $P\mathcal{S}$ being a $\pi$-tribe. 

We still need to check that the first component $\mathcal{T}' \to \mathcal{T}$ of the span is a trivial fibration. To see this, observe that it also arises as a pullback of a trivial fibration:

\[\begin{tikzcd}
	{\mathcal{T}'} && {P\mathcal{S}} \\
	\\
	{\mathcal{T} \times \mathcal{S}} && {\mathcal{S} \times \mathcal{S}} \\
	\\
	{\mathcal{T}} && {\mathcal{S}}
	\arrow[from=1-3, to=3-3]
	\arrow["{F \times id_{\mathcal{S}}}"', from=3-1, to=3-3]
	\arrow[from=1-1, to=3-1]
	\arrow["u", from=1-1, to=1-3]
	\arrow["\lrcorner"{anchor=center, pos=0.125}, draw=none, from=1-1, to=3-3]
	\arrow[from=3-1, to=5-1]
	\arrow[from=3-3, to=5-3]
	\arrow["F"', from=5-1, to=5-3]
	\arrow["\lrcorner"{anchor=center, pos=0.125}, draw=none, from=3-1, to=5-3]
\end{tikzcd}\]

Finally, by the universal property of the pullback, we get a morphism $m$ (which needs not preserve exponentials) as below
\[\begin{tikzcd}
	{\mathcal{T}} && {\mathcal{S}} \\
	& {\mathcal{T}'} && {P\mathcal{S}} \\
	\\
	& {\mathcal{T} \times \mathcal{S}} && {\mathcal{S} \times \mathcal{S}}
	\arrow[from=2-4, to=4-4]
	\arrow["{F \times id_{\mathcal{S}}}"', from=4-2, to=4-4]
	\arrow[from=2-2, to=4-2]
	\arrow["u", from=2-2, to=2-4]
	\arrow["\lrcorner"{anchor=center, pos=0.125}, draw=none, from=2-2, to=4-4]
	\arrow["{<id_{\mathcal{T}},F>}"', curve={height=6pt}, from=1-1, to=4-2]
	\arrow["m"', dashed, from=1-1, to=2-2]
	\arrow["F", curve={height=-6pt}, from=1-1, to=1-3]
	\arrow["{\iota_{\mathcal{S}}}", curve={height=-6pt}, from=1-3, to=2-4]
\end{tikzcd}\]
so that the following diagram commutes:

\[\begin{tikzcd}
	& {\mathcal{T}'} \\
	{\mathcal{T}} && {\mathcal{S}} \\
	& {\mathcal{T}}
	\arrow["\sim"{description}, from=1-2, to=2-1]
	\arrow[from=1-2, to=2-3]
	\arrow["{id_{\mathcal{T}}}", from=3-2, to=2-1]
	\arrow["f"', from=3-2, to=2-3]
	\arrow["m"', from=3-2, to=1-2]
\end{tikzcd}\]\end{proof}

We now turn to the result of interest of this section, which provides a local version of the previous lemma. We will be able to adapt it by using the characterization of the internal products in terms of the exponentials in the fibrant slices:
 
\begin{lemma}
\label{expprod}
Let $\mathcal{T}$ be a tribe. Then, $\mathcal{T}$ is a $\pi$-tribe if and only if all fibrant slices $\mathcal{T}(A)$, for $A$ an object of $\mathcal{T}$, admit exponentials, and  $x^z \to y^z$ is a fibration as soon as $x \to y$ is a fibration. 
 A morphism of tribes $F : \mathcal{T} \to \mathcal{S}$, between two $\pi$-tribes, preserves internal products up to isomorphism if and only if all the induced functors $F(A) : \mathcal{T}(A) \to \mathcal{S}(FA)$, for $A$ an object of $\mathcal{T}$, between the fibrant slices, preserve exponentials up to isomorphism. 
\end{lemma}

\begin{proof}
 Recall that, given a locally cartesian closed category $\mathcal{C}$, the dependent products of an arrow along another arrow can be expressed in terms of the exponentials in the slice category $\mathcal{C}_{/x}$, and conversely.
 Precisely, given two arrows $f : x \to y$ and $g : e \to x$, we have the following pullback square in $\mathcal{C}_{/y}$:

\[\begin{tikzcd}
	{\Pi_fg} && {[f,f \circ g]_{\mathcal{C}_{/y}}} \\
	\\
	{id_y} && {[f,f]_{\mathcal{C}_{/y}}}
	\arrow[from=3-1, to=3-3]
	\arrow[from=1-3, to=3-3]
	\arrow[from=1-1, to=3-1]
	\arrow[from=1-1, to=1-3]
	\arrow["\lrcorner"{anchor=center, pos=0.125}, draw=none, from=1-1, to=3-3]
\end{tikzcd}\]
where we have denoted the exponential in $\mathcal{C}_{/y}$ by $[-,-]_{\mathcal{C}_{/y}}$, and where the bottom map is the transpose of the identity arrow $f \to f$ along the adjunction defining the exponential $[f,-]_{\mathcal{C}_{/y}}$.

This formula can be used in the context of a tribe, when $f$ and $g$ are fibrations, and the resulting construction behaves as expected in the definition of a $\pi$-tribe.
 \end{proof}

\begin{lemma}
\label{strf}
 Let $F : \mathcal{T} \to \mathcal{S}$ be a morphism between $\pi$-tribes in $\mathbf{scTrb_{\pi,\sim}}$.

 Then there exists a span $\mathcal{T}' \to \mathcal{T} \times \mathcal{S}$ such that $\mathcal{T}'$ is a semi-cubical $\pi$-tribe, $\mathcal{T}' \to \mathcal{T}$ is a $\pi$-closed weak equivalence and $\mathcal{T}' \to \mathcal{S}$ is a $\pi$-closed functor, fitting in a commutative diagram:

\[\begin{tikzcd}
	& {\mathcal{T}} \\
	{\mathcal{T}} && {\mathcal{S}} \\
	& {\mathcal{T}'}
	\arrow["\sim"{description}, from=3-2, to=2-1]
	\arrow[from=3-2, to=2-3]
	\arrow["{id_{\mathcal{T}}}"', from=1-2, to=2-1]
	\arrow["F", from=1-2, to=2-3]
	\arrow["m"', from=1-2, to=3-2]
\end{tikzcd}\]
 Here $m$ is a weak equivalence; additionally, $\mathcal{T}' \to \mathcal{S}$ is a fibration, and $\mathcal{T}' \to \mathcal{T}$ is a trivial fibration.
\end{lemma}

\begin{proof}
 
The construction is the same as that of \cref{strf0}, and the proof is also similar. We use the notation $\mathcal{T}(A)$ to denote the fibrant slice of $\mathcal{T}$ under an object $A$, as introduced in Proposition 1.4.6 of \cite{joyal2017notes}. For every object $A$ of $\mathcal{T}'$, the pullback \eqref{pbr} yields a commutative square
 
\[\begin{tikzcd}
	{\mathcal{T}'(A)} && {P\mathcal{S}(uA)} \\
	\\
	{\mathcal{T}(\partial_0'A) \times \mathcal{S}(\partial_1'A)} && {\mathcal{S}(\partial_0\circ  u A) \times \mathcal{S}(\partial_1\circ  u A)}
	\arrow["{\partial(A)}", two heads, from=1-3, to=3-3]
	\arrow["{F(\partial_0'A) \times id_{\mathcal{S}(\partial_1'A)}}"', from=3-1, to=3-3]
	\arrow["{\partial'(A)}"', two heads, from=1-1, to=3-1]
	\arrow["{u(A)}", from=1-1, to=1-3]
	\arrow["\lrcorner"{anchor=center, pos=0.125}, draw=none, from=1-1, to=3-3]
\end{tikzcd}\]
which is a pullback (this can be checked by unfolding the definitions).

As before, $\mathcal{T}'(A)$ has exponentials, and, the projections $\mathcal{T}'(A) \to \mathcal{T}(\partial_0'A)$ and $\mathcal{T}'(A) \to \mathcal{S}(\partial_1\circ  u A)$ preserve those.
To see this, let $B = (b,x, v \to Fb \times x)$ and $C = (c,y, w \to Fc \times y)$  be two objects of the fibrant slice $\mathcal{T}'(A)$. Consider the exponential $u(A)B^{u(A)C}$ in $P\mathcal{S}(uA)$, which is of the form $(z \to F(b)^{F(c)} \times x^y)$.
  We claim that the exponential $B^C$ in $\mathcal{T}'(A)$ can be defined as the tuple $S = (b^c,x^y,s \to Fb^{Fc} \times x^y)$, where $s \to Fb^{Fc} \times x^y$ is constructed by pullback in $\mathcal{S}(\partial_0\circ  u A \times \partial_1\circ  u A)$ as follows:

\begin{equation}
\begin{tikzcd}
	s && z \\
	\\
	{F(b^c) \times x^y} && {F(b)^{F(c)} \times x^y}
	\arrow[from=3-1, to=3-3]
	\arrow[two heads, from=1-3, to=3-3]
	\arrow[two heads, from=1-1, to=3-1]
	\arrow[from=1-1, to=1-3]
	\arrow["\lrcorner"{anchor=center, pos=0.125}, draw=none, from=1-1, to=3-3]
\end{tikzcd}
\end{equation}

Checking that the object so defined can be equipped with an evaluation morphism that makes it an exponential object in $\mathcal{T}'(A)$ is done just like in the \cref{strf0}. By construction, the exponentials are preserved by the projections $\mathcal{T}'(A) \to \mathcal{T}(\partial_0'A)$ and $\mathcal{T}'(A) \to \mathcal{S}(\partial_1\circ  u A)$. The following still holds: if $B \to B'$ is a fibration in $\mathcal{T}'(A)$, then the induced map $B^C \to B'^C$ is also be a fibration. Applying Lemma \ref{expprod}, this concludes the proof of the statement.
 \end{proof}

Combining \cite[Lemma 6.1]{cherradi2026flat} and \cref{strf}, we get the following result:

\begin{lemma}
\label{r_strf}
 Let $F : \mathcal{T} \to \mathcal{S}$ be a functor between $\pi$-tribes preserving internal product up to equivalence. Assume that $F$ maps pullbacks along fibrations to pullback squares that are homotopy pullbacks, that $F(*_\mathcal{T}) \to *_\mathcal{S}$ is a weak equivalence and that weak equivalences are sent to weak equivalences by $F$. Then, in the following pullback, 
\[\begin{tikzcd}[ampersand replacement=\&]
	{\mathcal{T}'} \&\& {P\mathcal{S}} \\
	\\
	{\mathcal{T} \times \mathcal{S}} \&\& {\mathcal{S} \times \mathcal{S}}
	\arrow["u", from=1-1, to=1-3]
	\arrow["{<\alpha,\beta>}"', from=1-1, to=3-1]
	\arrow["\lrcorner"{anchor=center, pos=0.125}, draw=none, from=1-1, to=3-3]
	\arrow["{<\partial_0,\partial_1>}", from=1-3, to=3-3]
	\arrow["{F \times id_{\mathcal{S}}}"', from=3-1, to=3-3]
\end{tikzcd}\]
computed in $\mathbf{Cat}$, $\mathcal{T}'$ inherits a $\pi$-tribe structure and the two projection functors $\alpha : \mathcal{T}' \to \mathcal{T}$ and $\beta : \mathcal{T'} \to \mathcal{S}$ are $\pi$-closed morphisms of tribes.
\end{lemma}

The rigidification results from the \cite{cherradi2026flat} can now be specialized to the context of $\pi$-tribes.

\begin{proposition}
\label{rlan_cpi}
 Consider a diagram in $\mathbf{scTrb_{\pi,\sim}}$
\[\begin{tikzcd}[ampersand replacement=\&]
	{\mathcal{C}} \&\& {\mathcal{D}} \\
	{\mathcal{C}'}
	\arrow["F", from=1-1, to=1-3]
	\arrow["K"', from=1-1, to=2-1]
\end{tikzcd}\]
where $\mathcal{C'}$ and $\mathcal{D}$ are $\pi$-tribes, and where $K$ is a weak equivalence. Then there exists a span 
\[\begin{tikzcd}[ampersand replacement=\&]
	\& {\mathcal{C}''} \\
	{\mathcal{C}'} \&\& {\mathcal{D}}
	\arrow["{K'}"', from=1-2, to=2-1]
	\arrow["{F'}", from=1-2, to=2-3]
\end{tikzcd}\]
in $\mathbf{scTrb_{\pi}}$ representing the same morphism in the homotopy category $\mathbf{Ho}(\mathbf{scTrb_{\pi,\sim}})$.
\end{proposition}

\begin{proof}
 Applying \cite[Proposition 6.4]{cherradi2026flat} to the diagram
\[\begin{tikzcd}[ampersand replacement=\&]
	{\mathcal{C}} \& {\mathcal{D}} \& {\mathcal{P}(D)} \\
	{\mathcal{C}'}
	\arrow["F", from=1-1, to=1-2]
	\arrow["K"', from=1-1, to=2-1]
	\arrow["{\mathbf{y}}", from=1-2, to=1-3]
\end{tikzcd}\]
we obtain an extension, which we turn into a span
\[\begin{tikzcd}[ampersand replacement=\&]
	\& {\mathcal{C}''_0} \\
	{\mathcal{C}'} \&\& {\mathcal{P}(D)}
	\arrow[from=1-2, to=2-1]
	\arrow[from=1-2, to=2-3]
\end{tikzcd}\]
representing the left Kan extension of $\mathbf{Ho}_\infty(\mathbf{y} \circ F)$ along $\mathbf{Ho}_\infty(K)$. Since $K$ is a weak equivalence, the natural transformation involved in the left Kan extension is invertible, and, hence, $\mathcal{C''}_0 \to \mathcal{P}(\mathcal{D})$ factors through $\mathbf{R}(\mathcal{D})$. Moreover, $\mathcal{C}''_0$ is a $\pi$-tribe and the span is composed of $\pi$-closed morphisms of tribes since it has been constructed applying the construction from \cref{r_strf} (observing that $\mathbf{R}(\mathcal{C})$ is a $\pi$-tribe as soon as $\mathcal{C}$ is equivalent to a $\pi$-tribe). Forming the following pullback then yields a span in $\mathbf{Trb_{\pi}}$:
\[\begin{tikzcd}[ampersand replacement=\&]
	\&\& {\mathcal{C}''} \\
	\& {\mathcal{C}''_0} \&\& {\mathbf{Q}(\mathcal{D})} \\
	{\mathcal{C}'} \&\& {\mathbf{R}(\mathcal{D})} \&\& {\mathcal{D}}
	\arrow["\sim"{description}, two heads, from=1-3, to=2-2]
	\arrow[from=1-3, to=2-4]
	\arrow["\sim"{description}, from=2-2, to=3-1]
	\arrow[from=2-2, to=3-3]
	\arrow["\sim"{description}, two heads, from=2-4, to=3-3]
	\arrow[from=2-4, to=3-5]
\end{tikzcd}\]
Using that $\mathbf{scTrb_{\pi}} \to \mathbf{Trb_{\pi}}$ is a DK-equivalence, we replace this span by a span in $\mathbf{scTrb_{\pi}}$, proving the desired result.
\end{proof}

\section{Proof of the conjecture}

At this point, we have introduced most of the technical machinery that will enable us to deduce the part of the internal language conjecture, Conjecture \ref{ilcj}, that we seek to prove in this document. We now provide a quick overview of the argument.

We begin by observing that there is a factorization $$\catname{CompCat}_{\Sigma, \Pi_\text{ext}, \text{Id}} \to \catname{Trb_\pi} \to \catname{QCat}_{lcc}$$ through the category of $\pi$-tribes $\catname{Trb_\pi}$. It is enough to show that both functors are DK-equivalences.
Furthermore, Cisinski's theorem gives a powerful characterization of DK-equivalences between fibration categories. Our strategy is to replace the relative categories we are studying by fibration categories, where the "replacement" functor is directly shown to be a DK-equivalence. Namely, we replace the relative category of tribes $\catname{Trb}$ by $\catname{scTrb}$, and the category $\catname{Trb_\pi}$ by $\catname{scTrb_\pi^p}$.
There are two main steps to show that  $\catname{Trb_\pi} \to \catname{QCat}_{lcc}$ is a DK-equivalence. Firstly, we want to take advantage of the fact that $\catname{Trb} \to \catname{QCat}_{lex}$ is known to be a DK-equivalence, so that it induces equivalences of hom-spaces, together with our rigidification procedure from \cite{cherradi2022interpreting}, which constructs, from a given locally cartesian closed quasicategory, a corresponding $\pi$-tribe (Theorem 2.4 there). Secondly, we want to rigidify the morphisms between such $\pi$-tribes: we move from a setting where the internal product is preserved up to equivalence, to one where it is preserved up to isomorphism. This is notably where the rigidification tool, Lemma \ref{strf}, as well as \cref{rlan_cpi}

\subsection{The DK-equivalences between categories of tribes and their semi-cubical counterparts}

The proposition below, which connects relative categories of tribes with their semi-cubical counterpart, relies on the same argument as Proposition 3.12 in \cite{ks2019internal}, which was used to establish the DK-equivalence $\mathbf{scTrb} \to \mathbf{Trb}$. We recall the definitions and arguments relevant to the proof in the rest of this section.

\begin{definition}[\cite{ks2019internal}, Section 3]
 Let $\mathcal{T}$ be a tribe. The category of frames $\text{Fr}\mathcal{T}$ on $\mathcal{T}$ is defined as the category of homotopical diagrams in $\mathcal{T}$ of shape $\Delta_\sharp^{op}$, where $\Delta_\sharp$ is the homotopical category whose underlying category is the subcategory of $\Delta$ spanned by the monomorphisms (i.e, the semi-simplex category), and where all maps are taken to be weak equivalences.
 \end{definition}
 
	Kapulkin and Szumilo established in Theorem 3.7 of \cite{ks2019internal} that the category $\text{Fr}\mathcal{T}$ can be endowed with the structure of a semi-simplicial tribe, and that this construction induces a functor $\text{Fr} : \mathcal{T} \mapsto \text{Fr}\mathcal{T}$ from $\mathbf{Trb}$ to $\mathbf{scTrb}$. Moreover, they observed that, if $\mathcal{T}$ is already a semi-simplicial tribe, then any object $x$ induces a canonical frame $x^{\Delta_\sharp[-]}$ defined on objects by the cotensors $x^{\Delta^n}$, and similarly on morphism by cotensoring with the face maps of $\Delta_\sharp$. This defines a functor $\text{cfr}\mathcal{T} : \mathcal{T} \to \text{scFr}(\mathcal{T})$. However, this construction is only pseudonatural in $\mathcal{T}$; consequently, it cannot be used directly to provide a natural transformation between $\text{Fr} \circ i$ and $id_{\mathbf{scTrb}}$,  where $i : \mathbf{scTrb} \to \mathbf{Trb}$ is the inclusion. Therefore, to show that $i :  \mathbf{scTrb} \to \mathbf{Trb}$ is a DK-equivalence inverse to $\text{Fr}$, the authors of \cite{ks2019internal} introduce a further step: they define $\widehat{\text{Fr}} \mathcal{T}$ as (a variation of) the gluing construction  along the functor $\text{cfr}\mathcal{T}:\mathcal{T} \to \text{scFr}(\mathcal{T})$, that comes equipped with two projections $\widehat{\text{Fr}}\mathcal{T} \to \text{Fr}\mathcal{T}$ and $\widehat{\text{Fr}}\mathcal{T} \to \mathcal{T}$, now natural in $\mathcal{T}$.
 
 For our purposes, the construction $\widehat{\text{Fr}}$ is unfortunately insufficient because the projection $\widehat{\text{Fr}}\mathcal{T} \to \text{Fr}\mathcal{T}$ is not $\pi$-closed, so we introduce a variation $\tilde{\text{Fr}}\mathcal{T}$ of the construction such that the projection $\tilde{\text{Fr}}\mathcal{T} \to \text{Fr}\mathcal{T}$ is $\pi$-closed:
 \begin{definition}
 For $\mathcal{T}$ a semi-cubical tribe, we can consider the following pullback in $\mathbf{scTrb}$:

\[\begin{tikzcd}
	{\tilde{\text{Fr}}\mathcal{T} } && {P\mathcal{\text{Fr}\mathcal{T} }} \\
	\\
	{\mathcal{T} \times \text{Fr}\mathcal{T} } && {\text{Fr}\mathcal{T}  \times \text{Fr}\mathcal{T} }
	\arrow[from=1-1, to=1-3]
	\arrow[two heads, from=1-1, to=3-1]
	\arrow["\lrcorner"{anchor=center, pos=0.125}, draw=none, from=1-1, to=3-3]
	\arrow[two heads, from=1-3, to=3-3]
	\arrow["{\text{cfr} \times id_{\text{Fr}\mathcal{T} }}"', from=3-1, to=3-3]
\end{tikzcd}\]
This is an instance of \cref{strf}, since the canonical frame functor $\text{cfr}$ preserves internal product up to weak equivalence (as it is a weak equivalence between tribes, as remarked in \cite{ks2019internal}).
\end{definition}

Just like $\text{Fr}\mathcal{T}$, the category $\tilde{\text{Fr}}\mathcal{T}$ enjoys a semi-cubical tribe structure, and this construction also induces a functor $\tilde{\text{Fr}} : \mathcal{T} \mapsto \tilde{\text{Fr}}\mathcal{T}$ from $\mathbf{scTrb}$ to itself, as we prove in the next lemma. This functor comes equipped with two natural transformations $\tilde{\text{Fr}} \to id_\mathbf{scTrb}$ and $\tilde{\text{Fr}} \to \text{Fr} \circ i$, given by the projections, whose components are now $\pi$-closed morphisms of semi-cubical tribes and moreover weak equivalences (because $\text{cfr}\mathcal{T} : \mathcal{T} \to \text{scFr}(\mathcal{T})$ is a weak equivalence). 

\begin{lemma}
\label{frtilde}
 The mapping $\mathcal{T} \mapsto \tilde{\text{Fr}}\mathcal{T}$ induces a functor, and the two projections $\tilde{\text{Fr}} \to id_\mathbf{scTrb}$ and $\tilde{\text{Fr}} \to \text{Fr} \circ i$ are (strictly) natural in $\mathcal{T}$.
\end{lemma}

\begin{proof}
First, note that $\text{cfr} : \mathcal{T} \to \text{scFr}(\mathcal{T})$ is only pseudonatural in $\mathcal{T}$. Therefore, given a morphism $\mathcal{T} \to \mathcal{T}'$, we have the following diagram:

\[\begin{tikzcd}[ampersand replacement=\&]
	{\tilde{\text{Fr}}\mathcal{T} } \&\& {P\mathcal{\text{Fr}\mathcal{T} }} \\
	\&\&\& {P\mathcal{\text{Fr}\mathcal{S} }} \\
	{\mathcal{T} \times \text{Fr}\mathcal{T} } \&\& {\text{Fr}\mathcal{T}  \times \text{Fr}\mathcal{T} } \\
	\& {\mathcal{S} \times \text{Fr}\mathcal{S} } \&\& {\text{Fr}\mathcal{S}  \times \text{Fr}\mathcal{S} }
	\arrow[from=1-1, to=1-3]
	\arrow[""{name=0, anchor=center, inner sep=0}, curve={height=12pt}, dashed, from=1-1, to=2-4]
	\arrow[two heads, from=1-1, to=3-1]
	\arrow["\lrcorner"{anchor=center, pos=0.125}, draw=none, from=1-1, to=3-3]
	\arrow[""{name=1, anchor=center, inner sep=0}, from=1-3, to=2-4]
	\arrow[two heads, from=1-3, to=3-3]
	\arrow[from=2-4, to=4-4]
	\arrow[""{name=2, anchor=center, inner sep=0}, "{\text{cfr} \mathcal{T}\times id_{\text{Fr}\mathcal{T} }}", from=3-1, to=3-3]
	\arrow[from=3-1, to=4-2]
	\arrow[from=3-3, to=4-4]
	\arrow[""{name=3, anchor=center, inner sep=0}, "{\text{cfr}\mathcal{S} \times id_{\text{Fr}\mathcal{S} }}"', from=4-2, to=4-4]
	\arrow["\simeq"{description, pos=0.2}, shift left=5, draw=none, from=0, to=1]
	\arrow["\simeq"{description}, draw=none, from=2, to=3]
\end{tikzcd}\]
Here, we can lift the bottom isomorphism against the isofibration $P\text{Fr}\mathcal{S} \to \text{Fr}\mathcal{S}  \times \text{Fr}\mathcal{S}$ to get the one in the top triangle. Actually, to ensure functoriality, the choice for the lift  of an isomorphism $(\alpha,\beta)$ in $\text{Fr}\mathcal{S}  \times \text{Fr}\mathcal{S}$ is made as follows:

\[\begin{tikzcd}[ampersand replacement=\&]
	\& x \\
	\\
	y \&\& z \&\& x \\
	\\
	\&\&\& {y'} \&\& {z'}
	\arrow["f"', from=1-2, to=3-1]
	\arrow["g", from=1-2, to=3-3]
	\arrow["{id_x}", from=1-2, to=3-5]
	\arrow["\alpha"', from=3-1, to=5-4]
	\arrow["\beta"'{pos=0.3}, from=3-3, to=5-6]
	\arrow["{f'}"{pos=0.8}, dashed, from=3-5, to=5-4]
	\arrow["{g'}", dashed, from=3-5, to=5-6]
\end{tikzcd}\]
 Therefore, we get a mediating morphism $\tilde{\text{Fr}}\mathcal{T} \to \tilde{\text{Fr}}\mathcal{S}$ as below,

\[\begin{tikzcd}
	{\tilde{\text{Fr}}\mathcal{T} } && {P\mathcal{\text{Fr}\mathcal{T} }} \\
	& {\tilde{\text{Fr}}\mathcal{S} } && {P\mathcal{\text{Fr}\mathcal{S} }} \\
	{\mathcal{T} \times \text{Fr}\mathcal{T} } && {\text{Fr}\mathcal{T}  \times \text{Fr}\mathcal{T} } \\
	& {\mathcal{S} \times \text{Fr}\mathcal{S} } && {\text{Fr}\mathcal{S}  \times \text{Fr}\mathcal{S} }
	\arrow[""{name=0, anchor=center, inner sep=0}, from=1-1, to=1-3]
	\arrow[from=1-1, to=2-2]
	\arrow[two heads, from=1-1, to=3-1]
	\arrow["\lrcorner"{anchor=center, pos=0.125}, draw=none, from=1-1, to=3-3]
	\arrow[from=1-3, to=2-4]
	\arrow[two heads, from=1-3, to=3-3]
	\arrow[""{name=1, anchor=center, inner sep=0}, from=2-2, to=2-4]
	\arrow[from=2-2, to=4-2]
	\arrow["\lrcorner"{anchor=center, pos=0.125}, draw=none, from=2-2, to=4-4]
	\arrow[from=2-4, to=4-4]
	\arrow[""{name=2, anchor=center, inner sep=0}, "{\text{cfr} \mathcal{T}\times id_{\text{Fr}\mathcal{T} }}", from=3-1, to=3-3]
	\arrow[from=3-1, to=4-2]
	\arrow[from=3-3, to=4-4]
	\arrow[""{name=3, anchor=center, inner sep=0}, "{\text{cfr}\mathcal{S} \times id_{\text{Fr}\mathcal{S} }}"', from=4-2, to=4-4]
	\arrow["\simeq"{description}, draw=none, from=0, to=1]
	\arrow["\simeq"{description}, draw=none, from=2, to=3]
\end{tikzcd}\]
 where the square on the left commutes exactly (so that the two projections are natural in $\mathcal{T}$), and is functorial (because of the canonical choices for the lift against the isofibrations of the form $P\mathcal{\text{Fr}\mathcal{T} } \to \text{Fr}\mathcal{T}  \times \text{Fr}\mathcal{T}$ made above).
 \end{proof}

 These constructions are compatible with the $\pi$-tribe structures:
\begin{proposition}
\label{piframe}
 If $\mathcal{T}$ is a $\pi$-tribe, then $\text{Fr}\mathcal{T}$ is a semi-cubical $\pi$-tribe. Likewise, if $\mathcal{T}$ is a semi-cubical $\pi$-tribe, then so is $\tilde{\text{Fr}}\mathcal{T}$. Moreover, if $\mathcal{T} \to \mathcal{S}$ is $\pi$-closed, then so is the canonical functor $\tilde{\text{Fr}}\mathcal{T} \to \tilde{\text{Fr}}\mathcal{S}$.
\end{proposition}

\begin{proof}
 We already know that $\text{Fr}\mathcal{T}$ is a semi-cubical tribe. It is also a $\pi$-tribe by \cref{piR}, which proves the first part of the proposition.
 For the second part, the category $\tilde{\text{Fr}}\mathcal{T}$, which is already known to be a semi-cubical tribe, is moreover a $\pi$-tribe by \cref{strf}. Moreover, by construction of the internal product in  $\tilde{\text{Fr}}\mathcal{T}$, in the latter lemma, the morphism $\tilde{\text{Fr}}\mathcal{T} \to \tilde{\text{Fr}}\mathcal{S}$ is $\pi$-closed as soon as $\mathcal{T} \to \mathcal{S}$ is.
\end{proof}

We are now in a position to establish the following result.
\begin{proposition}
\label{semipi}
 The following three functors are DK-equivalences:
 \begin{align*}
  \mathbf{scTrb} & \to  \mathbf{Trb}\\
  \mathbf{scTrb_\pi} &\to \mathbf{Trb_\pi}\\
  \mathbf{scTrb_{\pi,\sim}} &\to \mathbf{Trb_{\pi,\sim}}
 \end{align*}
\end{proposition}

\begin{proof}
 
  We give the proof for the second functor, the arguments for the other two are completely analogous. As observed in \cref{piframe}, we have functors $\text{Fr} : \mathbf{Trb_\pi} \to \mathbf{scTrb}_\pi$ and $\tilde{\text{Fr}} : \mathbf{scTrb_\pi} \to \mathbf{scTrb}_\pi$.
  We can see that the first functor is a DK-equivalence inverse to the inclusion $i_\pi : \mathbf{scTrb_\pi} \to \mathbf{Trb}_\pi$; indeed:
  \begin{itemize}
   \item Evaluation at $[0]$ induces a natural weak equivalence $i_\pi \circ \text{Fr} \to id_{\mathbf{Trb}_\pi}$ by \cite[Theorem 3.3]{cherradi2026semi}.
   \item The two projections $\tilde{\text{Fr}} \to id_{\mathbf{scTrb}_\pi}$ and $\tilde{\text{Fr}} \to \text{Fr} \circ i_\pi$ define a zig-zag of natural weak equivalence between $id_{\mathbf{scTrb}_\pi}$ and $\text{Fr} \circ i_\pi$.
  \end{itemize}
 \end{proof}

\subsection{The DK-equivalence $\catname{scTrb_\pi^p}  \to \catname{scTrb_\pi}$}

We will denote the objects and arrows of the span-shaped homotopical category $\mathbf{Sp}_w$ as shown in the diagram below.

\[\begin{tikzcd}[ampersand replacement=\&]
	\& 01 \\
	0 \&\& 1
	\arrow["{\pi_0}"', from=1-2, to=2-1]
	\arrow["{\pi_1}", from=1-2, to=2-3]
\end{tikzcd}\]
We define a functor $P_\iota : (\square_\sharp^s)^{op} \times \mathbf{Sp}_w \to (\square_\sharp^s)^{op}$
by mapping:
\begin{itemize}
 \item Any object of the form $([n],01)$ to $[n+1] = I^1 \otimes I^n$
 \item Any object of the form $([n],0)$ to $[n]$
 \item Any object of the form $([n],1)$ to $[n]$
 \item Any map of the form $(id_{[n]},\pi_0)$ to $\delta_0^{op} \otimes id_{I^n} : I^{n+1} = I^1 \otimes I^n \to I^n$
 \item Any map of the form $(id_{[n]},\pi_1)$ to $\delta_1^{op} \otimes id_{I^n} : I^{n+1} = I^1 \otimes I^n \to I^n$
\end{itemize}
The rest of the action on morphisms is clear. The following property is straightforward to verify:

\begin{lemma}
 The functor $P_\iota$ is a discrete opfibration.
\end{lemma}

\begin{proof}
 The fiber above $n > 0$ is given by a three-element set, corresponding to $(n-1,01)$, $(n,0)$ and $(n,1)$. For $n = 0$, it is the two-element set given by $(0,0)$ and $(0,1)$. The opfibration property is easy to check.
\end{proof}

Given a tribe $\mathcal{T}$, $\text{scFr}(\mathcal{T})$ is a semi-cubical tribe by \cite[Theorem 3.3]{cherradi2026semi}. Therefore, we have a functor $$\iota_{\text{scFr}(\mathcal{T})} : \text{scFr}(\mathcal{T}) \to  P(\text{scFr}(\mathcal{T}))$$
obtained from the cotensor by $\square_\sharp^1$.

\begin{lemma}
 With the notation above, the functor $\iota_{\text{scFr}(\mathcal{T})}$ coincides with precomposition with $P_\iota$:
 $$P_\iota^* : \text{scFr}(\mathcal{T}) := \mathcal{T}^{(\square_\sharp^s)^{op}}_R \to  P(\text{scFr}(\mathcal{T})) := (\mathcal{T}^{(\square_\sharp^s)^{op}}_R)^{\mathbf{Sp}_w}_R$$
\end{lemma}

\begin{proof}
 This follows from the definition of the Day convolution derived from the monoidal structure on $\square_\sharp^s$ since $(\square_\sharp^s)^1$ corresponds to the generating object $I^1$. For this, it is important to note that $P(\text{scFr}(\mathcal{T}))$ is isomorphic to the category of homotopical Reedy fibrant diagram from $(\square_\sharp^s)^{op} \times \mathbf{Sp}_w$ to $\mathcal{T}$, where the product is given the canonical Reedy structure (here, it is actually an inverse category). This is because the Reedy fibrancy criterion is expressed in terms of limits, and limits commute with one another.
\end{proof}

\begin{corollary}
 The cubical frame functor $$\text{Fr} : \mathbf{Trb}_\pi \to \mathbf{scTrb}_\pi$$ factors through $\mathbf{scTrb^p}_\pi$
\end{corollary}

\begin{proof}
We have shown that $\iota_{\text{scFr}(\mathcal{T})}$ coincides with $P_\iota^*$, which is a $\pi$-closed morphism of tribes by \cref{piR}, since $P_\iota$ is a discrete opfibration.
\end{proof}

\begin{proposition}
\label{dkPtop}
 The inclusions $$\mathbf{scTrb^p}_\pi \to \mathbf{scTrb}_\pi \to \mathbf{Trb}_\pi$$ are DK-equivalences.
\end{proposition}

\begin{proof}
 Given the corollary above, the proof of \cref{semipi} can be refined to see that $\mathbf{scTrb^p}_\pi \to \mathbf{Trb}_\pi$ is a DK-equivalence. The rest follows by the $2$-out-of-$3$ property.
\end{proof}

\subsection{Preservation of the internal product: from up-to-equivalence to up-to-isomorphism}

Given a tribe $\mathcal{C}$, we use the notation $\mathbf{R}\mathcal{C}$ from \cite{ks2019internal} to denote the full subcategory of $\mathbf{SSet}^{L^H\mathcal{C}^{op}}$ spanned by the essentially representable simplicial presheaves that are fibrant. This category can be equipped with a canonical tribe structure (Theorem 6.10 in \cite{ks2019internal}).

\begin{lemma}
\label{pirc}
 Suppose that $\mathcal{C}$ is a tribe that is equivalent (i.e., connected by a zig-zag of weak equivalences of tribes) to a $\pi$-tribe $\mathcal{C}'$.
 Then $R\mathcal{C}$ is a $\pi$-tribe. 
\end{lemma}

\begin{proof}
For notational simplicity, we only check that $R\mathcal{C}$ has exponentials, but the argument applies equally to any internal product of a fibration along a fibration.
Consider two objects $A$ and $B$ of $R\mathcal{C}$.
Let $a$ and $b$ be two objects of $\mathcal{C}$ representing $A$ and $B$, in the sense that there exist weak equivalences $L^H\mathcal{C}(-,a) \to A$ and $L^H\mathcal{C}(-,b) \to B$. 
We fix an equivalence of categories $$\alpha : \mathbf{Ho}(\mathcal{C}) \simeq \mathbf{Ho}(\mathcal{C}')$$
Note that $\mathbf{Ho}(\mathcal{C}')$ is cartesian closed; hence so is $\mathbf{Ho}(\mathcal{C})$ by equivalence.
Consider an object $c$ of $\mathcal{C}$ such that $\alpha c$ is isomorphic to the exponential $(\alpha a)^{\alpha b}$ computed in the $\pi$-tribe $\mathcal{C}'$. 
Consider the tribe $P\mathcal{C}$ whose objects are the fibrant objects of the injective model structure on $\mathbf{SSet}^{L^H\mathcal{C}^{op}}$, and whose morphisms are the maps in $\mathbf{SSet}^{L^H\mathcal{C}^{op}}$ between any two fibrant objects. The tribe structure of $P\mathcal{C}$ comes, as usual, from defining a fibration as one in the injective model structure. $P\mathcal{C}$ is moreover a $\pi$-tribe because $\mathbf{SSet}^{L^H\mathcal{C}^{op}}$ is a locally cartesian closed model category.
Consider the exponential $A^B$ computed in $P\mathcal{C}$. We claim that it is essentially represented by $c$.

Because $\mathcal{C}'$ is a $\pi$-tribe, $(\alpha a)^{\alpha b}$ is an exponential in the homotopy category $$\mathbf{Ho}(\mathcal{C}') \simeq \mathbf{Ho}(\mathcal{C}) \simeq \mathbf{Ho}(R\mathcal{C})$$
But the homotopy category $\mathbf{Ho}(R\mathcal{C})$ is a full subcategory of $\mathbf{Ho}(P\mathcal{C})$, so that the exponential $A^B$ in $P\mathcal{C}$ is isomorphic to $(\alpha a)^{\alpha b}$ (modulo the identification given by the previous equivalences of homotopy categories).

The same kind of argument allow us to reason about morphisms, thanks to Lemma 6.7 in \cite{ks2019internal}, so as to extend the proof to internal products of fibrations along fibrations. Explicitly, consider two fibrations $F :  E \to A$ and $P : A \to B$ in $R\mathcal{C}$. Applying \cite[Lemma 6.7]{ks2019internal}, we have commutative squares 
\[\begin{tikzcd}[ampersand replacement=\&]
	{L^H\mathcal{C}(-,a)} \&\& A \&\& {L^H\mathcal{C}(-,e)} \&\& E \\
	\\
	{L^H\mathcal{C}(-,b)} \&\& B \&\& {L^H\mathcal{C}(-,a)} \&\& A
	\arrow["{r_a}", from=1-1, to=1-3]
	\arrow["{L^H\mathcal{C}(-,f)}"', from=1-1, to=3-1]
	\arrow["P", two heads, from=1-3, to=3-3]
	\arrow["{r_e}", from=1-5, to=1-7]
	\arrow["{L^H\mathcal{C}(-,p)}"', from=1-5, to=3-5]
	\arrow["F", two heads, from=1-7, to=3-7]
	\arrow["{r_b}"', from=3-1, to=3-3]
	\arrow["{r_a}"', from=3-5, to=3-7]
\end{tikzcd}\]
with $r_e$, $r_a$ and $r_b$ weak equivalences, and where $f : e \to a$ and $p : a \to b$ are fibrations in $\mathcal{C}$. Since we have equivalences of homotopy categories $$\alpha_b : \mathbf{Ho}(\mathcal{C}_{/b}) \simeq \mathbf{Ho}(\mathcal{C}'_{/b'})$$ and $$\alpha_a : \mathbf{Ho}(\mathcal{C}_{/a}) \simeq \mathbf{Ho}(\mathcal{C}'_{/a'})$$ where $p' : a' \to b'$ is a fibration in $\mathcal{C'}$ equivalent to the fibration $a \to b$ (in that they yield isomorphic arrows in the homotopy category $\mathbf{Ho}(\mathcal{C}) \simeq \mathbf{Ho}(\mathcal{C}')$). Similarly, we have a fibration $f' : e' \to a'$ in $\mathcal{C'}$, equivalent to $e \to a$. Now, we can form the internal product $\Pi_{p'} f'$ in $\mathcal{C'}$ since it is a $\pi$-tribe, and there is a fibration $x \to b$ in $\mathcal{C}$ equivalent to $\Pi_{p'} f' \to b'$. The internal product $\Pi_P F \to B$ computed in $P\mathcal{C}$ is then essentially represented by $x \to b$, as we can see by arguing with the derived adjunction:

\[\begin{tikzcd}[ampersand replacement=\&]
	{\mathbf{Ho}(R\mathcal{C}_{/B})} \&\& {\mathbf{Ho}(\mathcal{C}_{/b})} \&\& {\mathbf{Ho}(\mathcal{C'}_{/b'})} \\
	\\
	{\mathbf{Ho}(R\mathcal{C}_{/A})} \&\& {\mathbf{Ho}(\mathcal{C}_{/a})} \&\& {\mathbf{Ho}(\mathcal{C'}_{/a'})}
	\arrow["\sim"{description}, from=1-1, to=1-3]
	\arrow["\sim"{description}, from=1-3, to=1-5]
	\arrow["{(p')^*}"', curve={height=6pt}, from=1-5, to=3-5]
	\arrow["\sim"{description}, from=3-1, to=3-3]
	\arrow["\sim"{description}, from=3-3, to=3-5]
	\arrow["{\Pi_{p'}}"', curve={height=6pt}, from=3-5, to=1-5]
\end{tikzcd}\]
Hence, $R\mathcal{C}$ admits an internal product of $F$ along $P$.
\end{proof}

\begin{remark}
 As in \cite[Lemma 6.2]{cherradi2026flat}, we can work with the category $\mathbf{SSet}^{\mathcal{C}^{op}}$ of simplicial presheaves on $\mathcal{C}$ with the adequate model structure, instead of $\mathbf{SSet}^{L^H\mathcal{C}^{op}}$. We will abuse notation, and still write $R\mathcal{C}$ for the tribe of essentially representable fibrant simplicial presheaves, as the two tribes are canonically equivalent. The proof of \cref{pirc} works equally well, so that it is moreover a $\pi$-tribe as soon as $\mathcal{C}$ is equivalent to a $\pi$-tribe.
\end{remark}

The following result bridges the gap between the "up-to-weak-equivalence" setting and the "up-to-isomorphism" one.

\begin{proposition}
\label{wtoi}
 The forgetful functor $\catname{scTrb_\pi^p} \to \catname{scTrb_{\pi,\sim}}$ is a DK-equivalence between fibration categories.
\end{proposition}

\begin{proof}
 Since the forgetful functor $\catname{scTrb_\pi^p} \to \catname{scTrb_{\pi,\sim}}$ is an exact functor between fibrations categories, Cisinski's theorem characterizing DK-equivalences applies (\cref{cisinski_dke}). Thus, it is enough to check that this functor satisfies the approximation properties. The weak equivalences are reflected because they consists of the DK-equivalences in both cases. To check the condition (AP2), we consider a map $F : \mathcal{T} \to \mathcal{S}$ in $\catname{scTrb_{\pi,\sim}}$ such that $\mathcal{S}$ lies in $\catname{scTrb_\pi^p}$. Using \cref{rlan_cpi}, we compute a rigid left Kan extension of the composite of  $F$ with $\mathcal{Q}(\mathcal{T}) \to \mathcal{T}$ along the Yoneda embedding $\mathcal{Q}(\mathcal{T}) \to \mathbf{R}(\mathcal{T})$, yielding a span representing the same morphisms in $\mathbf{Ho}(\catname{scTrb_{\pi,\sim}})$,
\[\begin{tikzcd}[ampersand replacement=\&]
	{\mathcal{Q}(\mathcal{T})} \&\& {\mathcal{T}} \&\& {\mathcal{S}} \\
	\\
	{\mathbf{R}(\mathcal{T})} \&\&\&\& {\mathcal{T}'}
	\arrow[from=1-1, to=1-3]
	\arrow["\sim"{description}, from=1-1, to=3-1]
	\arrow["F", from=1-3, to=1-5]
	\arrow[from=3-5, to=1-5]
	\arrow["\sim"{description}, from=3-5, to=3-1]
\end{tikzcd}\]
where $\mathcal{T}'$ is a $\pi$-tribe.

By definition of $\mathbf{Ho}(\catname{scTrb_{\pi,\sim}})$, both spans are connected by a zig-zag of weak equivalences, so there exists a tribe $\mathcal{T}''$ and a homotopy commutative diagram as below.
\[\begin{tikzcd}[ampersand replacement=\&]
	{\mathcal{Q}(\mathcal{T})} \&\& {\mathcal{S}} \\
	\& {\mathcal{T}''} \\
	{\mathbf{R}(\mathcal{T})} \&\& {\mathcal{T}'}
	\arrow[from=1-1, to=1-3]
	\arrow["\sim"{description}, from=1-1, to=3-1]
	\arrow["\sim"{description}, from=2-2, to=1-1]
	\arrow[from=2-2, to=1-3]
	\arrow["\sim"{description}, from=2-2, to=3-1]
	\arrow["\sim"{description}, from=2-2, to=3-3]
	\arrow[from=3-3, to=1-3]
	\arrow["\sim"{description}, from=3-3, to=3-1]
\end{tikzcd}\]
Hence, there exists a  tribe $\mathcal{T}'''$ and a homotopy $H : \mathcal{T}''' \to P\mathcal{S}$ as in the following diagram,
\[\begin{tikzcd}[ampersand replacement=\&]
	\&\& {\mathcal{T}} \&\& {\mathcal{S}} \\
	\&\&\&\&\&\& {P\mathcal{S}} \\
	{\mathcal{T}'''} \&\&\& {\mathcal{T}} \& {\mathcal{S}} \\
	\& {\mathcal{Q}(\mathcal{T})} \\
	\&\& {\mathbf{R}(\mathcal{T})} \&\&\& {\mathcal{T}_\pi} \\
	\& {\mathcal{T}''} \&\& {\mathcal{T}'}
	\arrow[from=1-3, to=1-5]
	\arrow[from=2-7, to=1-5]
	\arrow[from=2-7, to=3-5]
	\arrow[curve={height=-12pt}, from=3-1, to=1-3]
	\arrow["H", curve={height=-12pt}, from=3-1, to=2-7]
	\arrow[from=3-1, to=6-2]
	\arrow[from=3-4, to=3-5]
	\arrow[from=4-2, to=3-4]
	\arrow[from=4-2, to=5-3]
	\arrow[from=5-6, to=2-7]
	\arrow["\ulcorner"{anchor=center, pos=0.125, rotate=180}, draw=none, from=5-6, to=3-5]
	\arrow[from=5-6, to=6-4]
	\arrow[from=6-2, to=4-2]
	\arrow[from=6-2, to=6-4]
	\arrow[from=6-4, to=3-5]
	\arrow[from=6-4, to=5-3]
\end{tikzcd}\]
where $\mathcal{T}_\pi$ is obtained by pullback. Therefore, considering the mediating arrow $\mathcal{T}''' \to \mathcal{T}_\pi$, we have a commutative diagram
\[\begin{tikzcd}[ampersand replacement=\&]
	{\mathcal{T}} \&\& {\mathcal{S}} \\
	\&\& {P\mathcal{S}} \\
	{\mathcal{T}'''} \&\& {\mathcal{T}_\pi}
	\arrow[from=1-1, to=1-3]
	\arrow[from=2-3, to=1-3]
	\arrow["\sim"{description}, from=3-1, to=1-1]
	\arrow["H", curve={height=-12pt}, from=3-1, to=2-3]
	\arrow["\sim"{description}, dashed, from=3-1, to=3-3]
	\arrow[from=3-3, to=2-3]
\end{tikzcd}\]
where the composite $\mathcal{T}_\pi \to \mathcal{S}$ lies in $\mathbf{scTrb}_\pi$. We can now apply the semi-cubical frame construction to this diagram, finally yielding a map $F_\pi$ in $\mathbf{scTrb}_\pi^p$ and a diagram witnessing that the approximation property (AP2) holds.
\[\begin{tikzcd}[ampersand replacement=\&]
	{\mathcal{T}} \&\& {\mathcal{S}} \\
	\\
	{\tilde{\text{scFr}}  \mathcal{T}} \&\& {\tilde{\text{scFr}} \mathcal{S}} \\
	\\
	{\tilde{\text{scFr}} \mathcal{T}'''} \&\& {\tilde{\text{scFr}}  \mathcal{T}_\pi}
	\arrow["F", from=1-1, to=1-3]
	\arrow["\sim"{description}, from=3-1, to=1-1]
	\arrow[from=3-1, to=3-3]
	\arrow[from=3-3, to=1-3]
	\arrow["\sim"{description}, from=5-1, to=3-1]
	\arrow["\sim"{description}, from=5-1, to=5-3]
	\arrow["{F_\pi}"', curve={height=24pt}, from=5-3, to=1-3]
	\arrow[from=5-3, to=3-3]
\end{tikzcd}\]
\end{proof}

\subsection{Final step in the proof of the conjecture}

In this final subsection, we will be putting the pieces back together to reach our initial goal.

By definition, a DK-equivalence $\alpha : R \to S$ between relative categories is a functor inducing an equivalence of simplicial sets between the hom-spaces and such that $\mathbf{Ho}(\alpha) : \mathbf{Ho}(R) \to \mathbf{Ho}(S)$ is essentially surjective on objects. In the next proposition, unlike our strategy so far, we will make use of this definition and investigate a map between hom-spaces, where the latter are computed as in the hammock localization. In particular, the vertices of the hom-spaces between two objects $X$ and $Y$ of $R$ are represented by zig-zags
\[\begin{tikzcd}
	X & \bullet & \bullet & \bullet & Y
	\arrow[from=1-1, to=1-2]
	\arrow[from=1-3, to=1-2]
	\arrow["{...}"{description}, draw=none, from=1-3, to=1-4]
	\arrow[from=1-4, to=1-5]
\end{tikzcd}\]
where the arrows are morphisms in $R$, that are moreover weak equivalences when they go backward (from right to left).

\begin{proposition}
\label{wtoqc}
 The functor $\mathbf{Ho}_\infty : \catname{scTrb_{\pi,\sim}} \to \catname{QCat}_{lcc}$ is a DK-equivalence.
\end{proposition}

\begin{proof}
 Given $X$ and $Y$ two objects of $\catname{scTrb_{\pi,\sim}}$, we first show that $\mathbf{Ho}_\infty$ defines a weak equivalence of simplicial set in the sense of the Quillen model structure between the hom-spaces: $$Hom_{\catname{scTrb_{\pi,\sim}}}(X,Y) \to Hom_{\catname{QCat}_{lcc}}(\mathbf{Ho}_\infty(X),\mathbf{Ho}_\infty(Y))$$
 We start by observing that the hom-space $Hom(X,Y)$ between $X$ and $Y$ in $\catname{scTrb_{\pi,\sim}}$ is a subspace of the hom-space $Hom'(X,Y)$, computed in $\catname{scTrb}$, whose vertices are the zig-zags of functors $F$ satisfying the property that $\mathbf{Ho}_\infty(F)$ is locally cartesian closed. Moreover, this subspace is given as a disjoint union of connected components. These connected components are characterized by the fact that their vertices are zig-zags involving only arrows in $\catname{scTrb_{\pi,\sim}}$. Indeed, consider a vertex in $Hom'(X,Y)$ which comes from a zig-zag in $\catname{scTrb_{\pi,\sim}}$, as in the top of the diagram below, and consider any other vertex in $Hom'(X,Y)$ which is connected to the first, as in the bottom of the diagram:

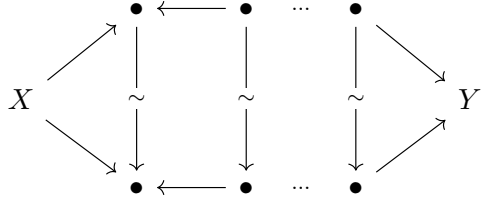
\[\begin{tikzcd}
	& \bullet & \bullet & \bullet \\
	X &&&& Y \\
	& \bullet & \bullet & \bullet
	\arrow[from=2-1, to=1-2]
	\arrow[from=1-3, to=1-2]
	\arrow["{...}"{description}, draw=none, from=1-3, to=1-4]
	\arrow[from=1-4, to=2-5]
	\arrow[from=2-1, to=3-2]
	\arrow[from=3-3, to=3-2]
	\arrow["{...}"{description}, draw=none, from=3-3, to=3-4]
	\arrow[from=3-4, to=2-5]
	\arrow["\sim"{description}, from=1-2, to=3-2]
	\arrow["\sim"{description}, from=1-3, to=3-3]
	\arrow["\sim"{description}, from=1-4, to=3-4]
\end{tikzcd}\]
In this situation, each arrow in the bottom zig-zag is the bottom arrow of a commutative square

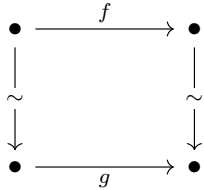
\[\begin{tikzcd}
	\bullet && \bullet \\
	\\
	\bullet && \bullet
	\arrow["f", from=1-1, to=1-3]
	\arrow["g"', from=3-1, to=3-3]
	\arrow["\sim"{description}, from=1-1, to=3-1]
	\arrow["\sim"{description}, from=1-3, to=3-3]
\end{tikzcd}\]
where the top arrow $f$ is known to lie in $\catname{scTrb_{\pi,\sim}}$. Therefore, $g$ is also such that $\mathbf{Ho}_\infty(g)$ is a locally cartesian closed $\infty$-functor. This proves that, if a vertex in a given connected component $\mathbf{C} \subset Hom'(X,Y)$ is in the image of $Hom(X,Y) \to Hom'(X,Y)$, then the connected component as a whole factors through the inclusion $Hom(X,Y) \to Hom'(X,Y)$. The same argument shows that, for $A$ and $B$ two objects of $\catname{QCat}_{lcc}$, the map $$Hom_{\catname{QCat}_{lcc}}(A,B) \to Hom_{\catname{QCat}_{lex}}(A,B)$$ is an inclusion of subspaces that are connected components. Note that, as such, this inclusion is trivially a Kan fibration.

 Now, the functor 
 $$\mathbf{Ho}_\infty: \catname{Trb} \to \catname{QCat}_{lex}$$
is a DK-equivalence (Theorem 9.10 in \cite{ks2019internal}),  and so is $\catname{scTrb} \to \catname{Trb}$, and thus the composite  $\mathbf{Ho}_\infty: \catname{scTrb} \to \catname{QCat}_{lex}$. Therefore, the induced morphisms between hom-spaces $$Hom_{\catname{scTrb}}(X,Y) \to Hom_{\catname{QCat}_{lex}}(\mathbf{Ho}_\infty(X),\mathbf{Ho}_\infty(Y))$$ is a weak equivalence of spaces. We are in the situation given by the following commutative square, which is actually a pullback:
 
\[\begin{tikzcd}
	{Hom_{\catname{scTrb_{\pi,\sim}}}(X,Y)} && {Hom_{\catname{QCat}_{lcc}}(\mathbf{Ho}_\infty(X),\mathbf{Ho}_\infty(Y))} \\
	\\
	{Hom_{\catname{scTrb}}(X,Y)} && {Hom_{\catname{QCat}_{lex}}(\mathbf{Ho}_\infty(X),\mathbf{Ho}_\infty(Y))}
	\arrow[from=1-1, to=1-3]
	\arrow["\sim"{description}, from=3-1, to=3-3]
	\arrow[two heads, from=1-1, to=3-1]
	\arrow[two heads, from=1-3, to=3-3]
	\arrow["\lrcorner"{anchor=center, pos=0.125}, draw=none, from=1-1, to=3-3]
\end{tikzcd}\]
 It follows that the top map is a weak equivalence of spaces, by right properness of the Quillen model structure on simplicial sets.
 
 We are, therefore, left to prove that the induced functor $$\mathbf{Ho}_\infty : \catname{scTrb_{\pi,\sim}} \to \catname{QCat}_{lcc}$$ between homotopy categories is essentially surjective on objects. But this has been established in \cite[Theorem 2.4]{cherradi2022interpreting}, so we are done.
 \end{proof}

 \begin{definition}
  We define $\catname{CompCat}_{\Sigma, \Pi_\text{ext}, \text{Id}}$ to be the relative category whose objects are the categorical models $\mathbf{C}$ of type theory defined as a variation of Definition 9.3 of \cite{ks2019internal}: we require the identity types to be strictly (and not just weakly) stable under substitutions, and for $\mathbf{C}$ to moreover admit $\Pi$-types with functional extensionality such that $\mathbf{ext} : \Pi_f P X \to P\Pi_f X$ maps the reflexive homotopy to the reflexive path, satisfying the $\Pi\text{-}\eta$ rule, and that are strictly stable under substitutions. The morphisms of $\catname{CompCat}_{\Sigma, \Pi_\text{ext}, \text{Id}}$ are the morphisms between comprehension categories that preserve the structure strictly.
  
 \end{definition}
 
 Given a comprehension category in $\catname{CompCat}_{\Sigma, \Pi_\text{ext}, \text{Id}}$, its base category has the structure of a tribe, where the fibrations are finite composites of context projections $\Gamma . A \to \Gamma$. When the comprehension category supports $\Pi$-types as in the previous definition, the resulting tribe is a $\pi$-tribe as seen in \cref{compcat_pitribe}. We now consider the functor $T_\pi : \catname{CompCat}_{\Sigma, \Pi_{\text{ext}}, \text{Id}} \to \catname{Trb_\pi}$.

 \begin{proposition}
 \label{dkcomp_trbpi}
  The functor $\catname{CompCat}_{\Sigma, \Pi_{\text{ext}}, \text{Id}} \to \catname{Trb_\pi}$ is a DK-equivalence.
 \end{proposition}

 \begin{proof}

    We use the same argument as in the proof of Theorem 9.9 in \cite{ks2019internal}. Explicitly, given a $\pi$-tribe $\mathcal{T}$ in $\catname{Trb_\pi}$, the canonical comprehension category $\mathcal{T}^{\rightarrow_\text{fib}} \to \mathcal{T}$ can be strictified by the left splitting functor in Definition 3.1.1 of \cite{lw2015local} as to yield an object $\mathcal{T}_!$ of $\catname{CompCat}_{\Sigma, \Pi_\text{ext}, \text{Id}}$. Indeed, as a $\pi$-tribe, $\mathcal{T}$ satisfies the condition (LF) (Definition 3.1.3 in \cite{lw2015local}), so we may apply Lemma 3.4.3.2 (resp. Lemma 3.4.2.4) of the same paper to deduce that $\mathcal{T}_!$ has strictly stable identity types (resp. dependent products).
    
    The functor
    $$C_\pi: \catname{CompCat}_{\Sigma, \Pi_\text{ext}, \text{Id}} \to \catname{Trb_\pi}$$ 
    defined by this construction is an inverse DK-equivalence to $T_\pi: \catname{CompCat}_{\Sigma, \Pi_\text{ext}, \text{Id}} \to \catname{Trb_\pi}$. Indeed, the composite $T_\pi \circ C_\pi : \catname{Trb_\pi} \to \catname{Trb_\pi}$ is the identity functor, and there is a natural weak equivalence $id_{\catname{CompCat}_{\Sigma, \Pi_\text{ext}, \text{Id}}} \to C_\pi \circ T_\pi$ by definition of category models of type theory (the objects of $\catname{CompCat}_{\Sigma, \Pi_\text{ext}, \text{Id}}$) whose component at $(C,T,\chi: T \to C^\rightarrow)$ is the weak equivalence $T \to C^{\rightarrow_\text{fib}} \to (C^{\rightarrow_\text{fib}})_!$.
 \end{proof}
 
 Putting everything together, we can finally conclude:
 
 \begin{theorem}
  The functor $\mathbf{Ho}_\infty : \catname{CompCat}_{\Sigma, \Pi_\text{ext}, \text{Id}} \to \catname{QCat}_{lcc}$ is a DK-equivalence.
  
 \end{theorem}

 \begin{proof}
  We have a commutative diagram
\[\begin{tikzcd}[ampersand replacement=\&]
	\&\&\& {\catname{scTrb_\pi}} \&\& {\catname{scTrb_\pi^p}} \\
	\\
	{\catname{CompCat}_{\Sigma, \Pi_\text{ext}, \text{Id}}} \&\&\& {\catname{Trb_\pi}} \\
	\\
	{\catname{QCat}_{lcc}} \&\&\& {\catname{Trb_{\pi,\sim}}} \&\& {\catname{scTrb_{\pi,\sim}}}
	\arrow["{\text{\cref{semipi}}}"{description}, from=1-4, to=3-4]
	\arrow["{\text{\cref{dkPtop}}}"{description}, from=1-6, to=1-4]
	\arrow["{\text{\cref{wtoi}}}"{description}, from=1-6, to=5-6]
	\arrow["{\text{\cref{dkcomp_trbpi}}}"{description}, from=3-1, to=3-4]
	\arrow[from=3-1, to=5-1]
	\arrow[from=3-4, to=5-4]
	\arrow["{\text{\cref{wtoqc}}}"{description}, from=5-4, to=5-1]
	\arrow["{\text{\cref{semipi}}}"{description}, from=5-6, to=5-4]
\end{tikzcd}\]
where the indicated arrows have already been shown to be DK-equivalences in the labeled propositions.
By the 2-out-of-3 property, we obtain that the functor $\catname{Trb_\pi}  \to \catname{Trb_{\pi,\sim}}$ is also a DK-equivalence. Finally, we conclude that $$\catname{CompCat}_{\Sigma, \Pi_\text{ext}, \text{Id}} \to \catname{QCat}_{lcc}$$ is a DK-equivalence as the composite of three DK-equivalences.
\end{proof}
\newpage

\begingroup
\setlength{\emergencystretch}{.5em}
\RaggedRight
\printbibliography
\endgroup
\end{document}